\pdfoutput=1
\documentclass[12pt]{article}
\usepackage[utf8]{inputenc}
\usepackage{geometry,amssymb,amsthm,amsmath,
graphics,enumerate}

\usepackage{authblk}
\usepackage{times}
\usepackage{amsfonts}
\usepackage{amssymb}
\usepackage{amscd}
\usepackage{url}

\usepackage{graphicx}

\usepackage{tikz-cd}

\usepackage{xargs}
\newfont{\msbm}{msbm10 scaled\magstephalf}
\def\aut{\rm Aut}

\def\rg{\mathop{\rm rg}}

\newtheorem{theorem}{Theorem}[subsection]

\newtheorem{corollary}[theorem]{Corollary}

\newtheorem{convention}[theorem]{Convention}
\newtheorem{lemma}[theorem]{Lemma}

\newtheorem{definition}[theorem]{Definition}

\newtheorem{claim}[theorem]{Claim}
\newtheorem{fact}[theorem]{Fact}

\newtheorem{remark}[theorem]{Remark}

\newtheorem{notation}[theorem]{Notation}

\newtheorem{example}[theorem]{Example}

\def\b1K{\mbox{\boldmath  K }_{-1}}

\def\bK{\mbox{\boldmath  K }}

\def\bL{\bf  L }

\def\Mscr{{\MM }}

\def\Rscr{{\cal R }}

\def\Ascr{{\mathcal A }}

\def\Lscr{{\cal L     }}

\def\Sscr{{\cal S   }}

\def\dom{\unrhd}

\def\Th{\mathop{\rm Th}}
\def\th{\mathop{\rm Th}}

\def\<{\langle}

\def\>{\rangle}

\def\id{\rm id}
\def\ker{\rm ker}

 \def\dom{\mathop{\rm dom}}

\def\isom{{\rm isom}}
\def\sl{\mathop{\rm SL}}
\def\psl{\mathop{\rm PSL}}
\def\gl{\mathop{\rm GL}}

\def\dcl{{\rm dcl}}

\def\acl{{\rm acl}}

\def\cl{{\rm cl}}
\def\pr{{\rm pr}}

\newbox\noforkbox \newdimen\forklinewidth
\setbox1\hbox to \wd0{\hfil\vrule width \forklinewidth depth-2pt
 height 10pt \hfil}
\wd1=0 cm \setbox\noforkbox\hbox{\lower 2pt\box1\lower
2pt\box0\relax}
\def\unionstick{\mathop{\copy\noforkbox}\limits}
\def\nonfork_#1{\unionstick_{\textstyle #1}}
\newbox\doesforkbox
\setbox\doesforkbox\hbox{\lower 2pt\box1 \lower
2pt\box2\lower2pt\box0\relax}

\def\aut{\rm Aut}

\newcommand{\forkindep}[1][]{%
  \mathrel{
    \mathop{
      \vcenter{
        \hbox{\oalign{\noalign{\kern-.3ex}\hfil$\vert$\hfil\cr
              \noalign{\kern-.7ex}
              $\smile$\cr\noalign{\kern-.3ex}}}
      }
    }\displaylimits_{#1}
  }
}
\newcommand{\nonforkindep}[1][]{%
  \mathrel{
    \mathop{
      \vcenter{
        \hbox{\oalign{\noalign{\kern-.3ex}\hfil$\vert$\rlap{$'$}\hfil\cr
              \noalign{\kern-.7ex}
              $\smile$\cr\noalign{\kern-.3ex}}}
      }
    }\displaylimits_{#1}
  }
}

\newcommand{\GG}{\mbox{\msbm G}}
\newcommand{\HH}{\mbox{\msbm H}}
\newcommand{\RR}{\mbox{\msbm R}}
\newcommand{\TT}{\mbox{\msbm T}}
\newcommand{\UU}{\mbox{\msbm U}}

\newcommand{\ZZ}{\mbox{\msbm Z}}

\newcommand{\QQ}{\mbox{\msbm Q}}

\newcommand{\NN}{\mbox{\msbm N}}

\newcommand{\MM}{\mbox{\msbm M}}
\newcommand{\CC}{\mbox{\msbm C}}
\newcommand{\Fscr}{\mathcal {F}}
\newcommand{\Gscr}{\mathcal {G}}

\def\sub'm{\prec_{\bK'}}
\def\grpf #1 #2{{\rm grp}_{#2}(#1)}
\def\spanf #1 #2{{\rm span}_{#2}(#1)}
\def\fldf #1 #2{{\rm fld}_{#2}(#1)}
\def\dclf #1 #2{{\rm dcl}_{#2}(#1)}
\def\rclf #1 #2{{\rm rcl}_{#2}(#1)}
\def\aclf #1 #2{{\rm acl}_{#2}(#1)}
\def\acff #1 #2{{\rm acf}_{#2}(#1)}
\def\strf #1 #2{{\rm str}_{#2}(#1)}
\def\tclf #1 #2{{\rm acf}_{#2}(#1)}

\def\abar{{\bf a}}

\def\dbar{{\bf d}}

\def\gbar{{\bf g}}
\def\hbar{{\bf h}}

\def\pbar{{\bf p}}
\def\qbar{{\bf q}}

\def\vbar{{\bf v}}

\def\xbar{{\bf x}}
\def\ybar{{\bf y}}

\def\tp{{\rm tp }}

\date{\today}

\newcommand{\im}{\mathrm{im}}

\title{Zilber's notion of logically perfect structure: Universal Covers}

\author[1]{John T. Baldwin}
\author[2]{Andrés Villaveces}
\affil[1]{Department of Mathematics, Statistics and Computer Science,

The University of Illinois at Chicago}
\affil[2]{Departamento de Matemáticas, Universidad Nacional de Colombia, Bogotá}

\begin{document}
\date{\today}
\maketitle

\abstract{We sketch the mathematical back ground and  the main ideas in the
proofs of categoricity of theories of several examples of universal covers --
reducing an analytic to a model theoretic (discrete) description. We hope this
discussion will be useful to a wide spectrum of mathematicians ranging from
those working in geometry to those working in logic; specifically, model
theory. }

\tableofcontents

\section{Introduction}\label{intro}



The goal of this paper is to sketch (hopefully for a wide spectrum of mathematicians ranging from
those working in geometry to those working in logic; specifically, model theory) some recent
interactions between model theory and a roughly
150-year old study of analytic functions involving complex analysis, algebraic topology, and number
theory that explore the canonicity of universal covers.
Towards this goal we  discuss and present several examples indicating the main ideas of the proofs
and the necessary changes in method for different situations.

Here is Zilber's description of his own project (from his 2000 Logic Colloquium talk in Paris~\cite{Zilberparis}):

\begin{quote}
\emph{
The initial hope of this author in  \cite{Zilbericm} that any
uncountably categorical structure comes from a classical context (the
trichotomy conjecture), was based on the belief that logically
perfect structures could not be overlooked in the natural progression
of mathematics. Allowing some philosophical license here, this was
also a belief in a strong logical predetermination of basic
mathematical structures.  As a matter of fact, it turned out to be
true in many cases.  \ldots
Another situation where this
principle works is the context of o-minimal structures \cite{PetStartri}. }

%

\end{quote}


A rather ambitious project aimed at finding \emph{categorical}
axiomatizations (De\-fi\-ni\-tion~\ref{catdef}) of various kinds of
\emph{universal covers} has been unfolding in the 21st century. The simplest
example of such universal covers is given by the short exact sequence:

\begin{eqnarray}0 \to \ker(\exp) \to (\CC,+,0) \ {\stackrel{exp}{\to}}
 (\CC,+,\cdot, 0,1)\to 1.
\end{eqnarray}\label{eqcover1}

Zilber's original project really aimed to understand the sequence
\begin{eqnarray}0 \to \ker(\exp) \rightarrowtail (\CC,+,\cdot,\exp) \         {\stackrel{exp}{ \twoheadrightarrow}}
(\CC,+,\cdot,\exp)\to 1.
\end{eqnarray}
The first diagram describes a two-sorted \emph{cover} of the multiplicative
group by the additive group. The {\em full} field structure is studied on the
range space although the kernel is of the homomorphism from $(\CC,+,0)$ to
$(\CC,\cdot,1)$.

The second~\cite{Zilberpseudoexp} corresponds to the \emph{theory} of the
complex exponential field. The  domain and range of the map are the same
exponential field but the kernel is again computed with respect to the
homomorphism $\exp$ from $(\CC,+)\rightarrow (\CC^*,\times)$.

In both cases, first order axioms are supplemented by an
$L_{\omega_1,\omega}$-sentence asserting the kernel is isomorphic to $\ZZ$,
i.e., is standard.
Here, we focus on three main \emph{families} of generalizations (described in
the chart below) of the first diagram.  As  this question was extended to
more general algebraic contexts, the fundamental cover diagram from
equation~(\ref{eqcover1})  changed to this more general situation:


\begin{eqnarray}
 C\ {\stackrel{p}{\twoheadrightarrow}}
S(\CC).
\end{eqnarray}

Notice two things:
\begin{itemize}
    \item The map  $p$  remains a projection, but it will
	significantly change as the family of examples unfolds. Also,
    \item there is no longer a kernel when $S(\CC)$ is not a group.
\end{itemize}

Therefore, in a rather Protean way, the infinitary description that in the particular case described
a `standard kernel' assumes various guises for different examples. Usually, the
descriptions are of `standard fibres' rather than  having a `standard kernel'. 

Crucially, in all cases except part of \S ~\ref{section:notop} the target
will be some kind of definable set in an algebraically closed field.  The
necessary vocabulary for the domain will vary among the situations
considered. Shimura varieties require a more general domain:

\begin{notation}\label{gs} (The general situation)
\begin{eqnarray}
 X^+\ {\stackrel{p}{\to}}
S(\CC) \to 1.
\end{eqnarray}
\end{notation}

Here, $S(\CC)$ is a variety  arising as  the quotient of the action of a discrete
group on $\HH$ (hyperbolic space)  or more generally (Shimura varieties) on  a  hermitian symmetric
domain $X^+$. The target is described by a first order  theory $T:= \th(S(\CC))$ in a large enough
 (field) countable vocabulary with
quantifier elimination (possible, as $S$ is definable in $(\CC,+,\times)$).
Notation~\ref{gs} thus instantiates the general schema, with
appropriate notations for specific cases to be given as we discuss them. Zilber describes the value of his
project in terms  of   `a complete formal invariant' (Remark~\ref{cfi}).

%
\begin{quote}
\emph{
The geometric value of the project is perhaps in the fact that the
 formulation of the categorical theory of the universal cover of a variety  $X$ \ldots
  is essentially a formulation of a complete formal
invariant of  $X$. }

\cite[1]{DawZil0}
\end{quote}

The following chart organizes the papers which are the major source for this study. It also provides
a keyword describing the main method or context used, and the section of this paper where issues
around the specific variant are explained.

\begin{center}

{\footnotesize
\begin{tabular}{|c| c   |c   |l |c|}
\hline
 & topic & paper & method/context & section \\
\hline
    1 & Complex exponentiation & \cite{Zilbercatex} & quasiminimality & §\ref{intro}\\
\hline \hline
2 & cov mult group & \cite{Zilbercovers} & quasiminimality & §\ref{intro}\\
3 & & \cite{BaysZil} & quasiminimality & \\

\hline
4 & $j$-function &\cite{Harristhe} & background & §\ref{aroundjfunction} \\
\hline
5 & Modular/Shimura Curves & \cite{DawHarris} & quasiminimality & §\ref{modshicurves} \\
6 &  Modular/Shimura Curves &\cite{DawZil0} &quasiminimality &\\
 \hline \hline
7 & finite Morley rank groups & \cite{BGH} & fmr \& notop  & §\ref{2sort}\\
8 & Abelian Varieties & \cite{BHP} & fmr \& notop /quasiminimality & §\ref{abelianvarieties}\\
\hline
9 &   Shimura {\em varieties} & \cite{Etrev} & notop & §\ref{Shivar}\\
\hline\hline
10 & Smooth varieties & \cite{Zilbomin} & o-quasiminimality & §\ref{smoothvar}\\
\hline\hline
\end{tabular}
}

\end{center}





\medskip
In this chart, the first line \cite{Zilberpseudoexp} (an axiomatization of the exponential  map from
the complex field to itself) differs from the others
in the role of the quantifier `there exists uncountably many'. In that case
it is essential to directly control the \textit{cardinality of the algebraic
closure of a countable set}. Moreover, in line 1 the domain has a field
structure that disappears in the two-sorted approach of the rest. In the
other lines of the chart, the infinitary logic $L_{\omega_1,\omega}$ is used
to control the size of fibers of the cover or when the structure is a group
the size of the kernel. This requirement suffices to also control the
cardinality of the algebraic closure. Lines 2-6   deal with \textit{curves}
($1$-dimensional objects) where categoricity is obtained by quasiminimality.
The third big horizontal block deals with higher dimensional varieties. Lines
7 and 9 stray from formal categoricity towards more traditional descriptions
of models;  quasiminimality is replaced by a different version of excellence
arising in Shelah's study of notop theories (an important notion in
Classification Theory). Both quasiminimality and `notop' apply to line 8.
The last line considers  families of covers of arbitrary smooth algebraic
varieties with an infinitary logic construction defined over o-minimal
expansions of the reals. There, the focus is on categoricity in $\aleph_1$.


It is worth noting that we could have organized our chart under a totally
different scheme.  The Abelian varieties and $(\CC,+)$ are specific
varieties. The $j$-function and the Shimura varieties may be regarded as
moduli spaces for (generalized) families of varieties\footnote{Various types
of Shimura varieties include Siegel, PEL-type, and Hodge-type; only some
parameterize algebraic avarieties.}. After preliminary discussions on the
model theoretic framework, in Section~\ref{modshicurves} we sketch in some
detail categoricity of universal covers of modular curves. In the later
sections we describe the modifications to this program necessary for higher
dimensions.
%


\subsection{Mathematical Encounters}

\subsubsection{Some ancient history: In and out of the Zilber world}


The first author turns to the first person singular for some memories:

Zilber and I both received our Ph.D.'s in the early 1970's. An important result
appeared in both theses: the solution to Morley's conjecture that an $\aleph_1$-categorical theory
has
finite Morley rank. Such an overlap
was not an issue during the Cold War. (On the other hand, Baldwin's advisor,
Lachlan, had to write an entirely new thesis when the result of the proposed one appeared in the west as he was about
to submit.)

I first (given my zero knowledge of Russian) learned in any detail of Zilber's work during the
1980-81 model theory year in Jerusalem. Greg Cherlin had no such deficiency and gave with Harrington
and Lachlan an alternate proof of Zilber's theorem that there were no finitely axiomatizable totally
categorical theories. They relied on  the classification of finite simple groups. A few
years later Boris completed his  model theoretic proof of the key combinatorial lemma avoiding that reliance.

I first knew Boris in any depth during the model theory semester in Chicago 91-92.  Unfortunately, I
had partially financed a semester by agreeing to be acting head the Fall semester, thereby
restricting my mathematical activity. In that busy fall, Boris and Angus Macintyre lectured on Tuesday's on
Zariski  geometries
and o-minimality, respectively.
The lively group include Macintyre, Zilber, Laskowski,  Marker, Otero, D'Aquino and myself, with Pillay
driving in weekly from Notre Dame.
Lunch was at a deli that Boris insisted on because of the soup followed by
coffee at Jamoch's, the first modern coffee house in the UIC area.

About that time, I began work on the Hrushovski construction, but in a quite
different direction from Boris: predimension with irrational $\alpha$. This
led to my work with Shelah giving the first full proof of the $0-1$ law with
edge probability $n^{-\alpha}$ and that
the theory of the Shelah-Spencer graph was
stable, building on the 1992 Ph.D. thesis of my student Shi. And
this
led to work
 with Kitty Holland on
fusions, giving the first construction of a rank $2$ field with a definable
infinite predicate. And then back to Boris and his work on complex
exponentiation.  Understanding his notion of quasiminimal excellence inspired
the desire to understand Shelah's more general notion of excellence. Thence
came my monograph on abstract elementary classes and subsequent work on
infinitary logic.
 In any case, visits several times a decade to Oxford always were exciting sources of ideas and
pleasant times.

\subsubsection{An unlikely encounter of two areas: MAMLS at Rutgers, 2001}

The second author of this paper witnessed and participated in one of those
momentous encounters of two areas that only seldom happen: during the MAMLS
Meeting at Rutgers in February 2001, a group of people working in Abstract
Elementary Classes (including Rami Grossberg, Monica VanDieren, Olivier
Lessmann and the second author of this paper) was very busy discussing
Shelah's notion of excellence, originally linked to his work in the model
theory of $L_{\omega_1,\omega}$. The \emph{$n$-amalgamation diagram} was very
much part of that discussion. There was a lecture by Boris Zilber at the end
of the day, and we all attended, not expecting to understand much, but eager
to see him speak. To our great surprise, at the end of Zilber's lecture
(dealing with exponential covers, mentioning many analytic number theoretic
methods that were arcane to us, and mixing in areas such as ``Nevanlinna
Theory'', etc.), he asked a final question and drew a picture underscoring
his question. Boris's picture was \emph{exactly} the $n$-amalgamation diagram
we had been discussing thoroughly with the AEC people those very same days;
Boris's question was exactly about the behaviour of types in the amalgam and
how it could be controlled by small pieces in the components. We jumped to
talk to him at the end of his lecture, with the excitement of seeing a
potential connection. Boris said he didn't know the model theory of
$L_{\omega_1\omega}$ but he would look into excellence\ldots

The rest is history: after a few weeks, a first draft of a proof of properties of
pseudoexponentiation drawing on a version of excellence and quasiminimality in $L_{\omega_1\omega}$
was circulated, and Zilber started using many methods from excellent classes
and infinitary logic. The richness of this approach has provided many interesting connections; we
explore some of them in our paper.

\subsection{A word of thanks from the second author}

Here, the second author turns to the first person singular, for this excerpt:
\begin{quotation}
    \emph{I would like to thank Boris Zilber, at a very personal level, for a
    life-changing conversation we had in 2007 in Utrecht, during a meeting organized by Juliette
    Kennedy, on connections between Mathematics, Philosophy and Art.
One evening, after dinner, Boris said ``let's go
for a walk and speak a bit about mathematics.'' In the cold night along the canals, he described,
for about an hour, some of what he had
been doing---I kept asking and asking questions. At some point, on a bridge, he turned to me and
said: ``But you, in what have you been working?'' I tried to gather my thoughts on the spot while
walking, and started describing a project we had back then, with Berenstein and
Hyttinen~\cite{BerensteinHyttinenVillaveces}, of
understanding independence notions in continuous logic, trying to extend the work of Chatzidakis and
Hrushovski to the continuous case, and encountering difficulties. Boris asked me to describe briefly
continuous model theory and continuous abstract elementary classes.
At some point, he said I obviously had tools for dealing with model
theoretical approaches to quantum mechanics. I asked how so. He said ``look at Gelfand triples,
\ldots''. I returned to Helsinki where I was spending a sabbatical, and Boris's remarks made a deep
change in my own approach to model theory, in the possibilities I started slowly unfolding. I am
deeply grateful for that momentous conversation, and for all the lines of work that have derived
from that evening!}

\hfill Andrés Villaveces
\end{quotation}

The authors want to thank many people who helped this project go through.
Among them, hoping not to forget important people, are, most notably
Sebastián Eterovi\'{c}, Jim Freitag, Jonathan Kirby, Anatoli Libgober, Ronnie
Nagloo, and Boris Zilber. Without their attention to our discussions, online,
at conferences, and on campus, this project would have been much harder to
complete.  The first author especially wants to thank Ronnie and Sebastián
for hours of conversation. The second author especially thanks Alex Cruz and
Leonardo Cano for many helpful discussions related to these subjects in the
Bogotá seminar before this project started. Finally, discussions with Thomas
Kucera and Martin Bays were very important at earlier stages of the
construction of this paper. Finally, the referee reports were invaluable.

\section{Model theory in Mathematics }\label{mtbg}

We first deal with some variations in  model theoretic and geometric terminology.

\subsection{Model theoretic background}
Mathematical logic makes a central distinction between a vocabulary and a
collection of sentences in a logic. For this reason, we use `language' only
for the second and reserve `vocabulary' for what is sometimes called
similarity type.

\begin{definition}[Vocabulary and Structure]\label{vocstr}
\begin{enumerate}
\item A vocabulary  $\tau$  is a collection of constant, relation, and function
symbols (with finitely many arguments).
\item A  $\tau$ -structure is a set in which each  $\tau$-symbol is
    interpreted, e.g., an $n$-ary relation symbol as an $n$-ary relation.
\end{enumerate}
\end{definition}


\begin{definition}{\em Full formalization} involves the following
components.

\begin{enumerate}
    \item A \textbf{vocabulary}  with associated notion of structure as in Definition~\ref{vocstr}.
    \item A \textbf{logic} $\Lscr$ has:
\begin{description}
\item[a] A class  $\Lscr(\tau)$
    of `well formed' \textbf{formulas}.
\item[b] A notion of `\textbf{truth} of a formula' from the class $\mathcal{L}\left( \tau \right) $  in a
    $\tau$-structure, usually denoted $\mathfrak{A}\models \varphi$.
\item[c] A notion of a ``formal \textbf{deduction}'' for this logic.
     \end{description}
\item \textbf{Axioms}: Specific sentences of the logic that specify the basic properties of the situation in question.
\end{enumerate}
\end{definition}
\begin{example} (Three important logics.)
\begin{enumerate}
\item The \textbf{first order language} $\Lscr_{\omega,\omega}(\tau)$ associated with  $\tau$ is the least set of formulas
containing the atomic  $\tau$ -formulas  and closed under \emph{finite}
Boolean operations and quantification over finitely many
individuals.

\item The  \textbf{$\Lscr_{\omega_1,\omega}(\tau)$  language} associated with  $\tau$  is the least set of                    formulas
containing the atomic  $\tau$ -formulas  and closed under \emph{countable} Boolean operations and quantification over finitely
many individuals.

\item The \textbf{second order    language} associated with  $\tau$, denoted $\Lscr^2(\tau)$,  is the least set of formulas extending   $\Lscr_{\omega,\omega}(\tau)$  by allowing quantification over sets and relations. $\Lscr^2(\{=\})$ is symbiotic (`morally equivalent', roughly speaking) with set theory.
    \end{enumerate}

\end{example}

Morley rank (corresponding to the Krull/Weil dimension in the particular case of fields) was introduced in
\cite{Morley65} to study theories categorical
in uncountable power.
Section~\ref{section:notop}   explores  the role of  finite Morley rank
groups in studying covers. Three good sources for the more advanced model
theory used here are \cite{Markerbook,TentZiegler,Poizatbook}.

\subsection{Various Viewpoints}

We now discuss two quite different uses of the three words \emph{automorphism}, \emph{model} and
\textit{definable},
coming from areas of mathematics relevant to this paper. (The difference in use
depending on the area of mathematics has been at times a source of confusion.)


\begin{remark}{\rm (Automorphism: two notions)}\label{twonot}

\begin{description}
\item [In Model Theory:] An \textbf{automorphism} of a  $\tau$-structure  $\Ascr$  is
a permutation of its universe $A$ that preserves (in both directions) each relation or function
symbol for  $\tau$. For instance, the automorphisms of a geometry (when given in terms of lines and
points together with an incidence relation) are the \emph{collineations}.
\item [In Algebraic Geometry:] An \textbf{automorphism} of  a variety is an
    invertible morphism\footnote{This begs the question of defining
    morphism. A good approximation is `definable map'. In algebraic
    geometry a morphism is (cf \cite[p 79: section 4.4]{Poizatbook2}) a
    constructible (generically quasi-rational) bijection. Biregular and
    birational are more specific syntactic restrictions on an
    isomorphism.}.
\end{description}

\end{remark}





\begin{remark}\label{canmod}{\rm (Model: two notions)}{\rm
\begin{description}
\item [In Model Theory:]

The word \textit{model} also sees different uses depending on the area. In
logic, a model is sometimes just a $\tau$-structure  but often signifies
that the structure satisfies a theory (as in `$\left(
\mathbb{C},+,\cdot,0,1 \right) $ is a model of the theory $ACF_0$'). {\em
Minimal model} might mean `no proper elementary submodel' or, very
differently, `every definable subset is finite or cofinite'.

\item [In Algebraic Geometry:] A \textbf{model} is a specific biregularity
    class within a birational equivalence class. In  Weil/Zariski style,  a
    variety is determined by a coordinate ring, but only up to isomorphism
    of this coordinate ring.  A `model' of the variety might be a specific
 affine variety with that coordinate ring, but any biregularly
isomorphic variety would also be a model.


Thus, unlike
model theory, algebraic geometry does not identify `models' up to
isomorphism. Rather, it looks  for   a specific `canonical representation'
among `isomorphic solution sets'. A {\em minimal model} is a smooth variety
$X$ with function field $K$ such that if $Y$ is another smooth variety with
function field $K$ and $f\colon X\mapsto Y$ is birational, then $f$ is an
isomorphism.
\end{description}}
\end{remark}

\begin{remark}{\rm (Definable/defined: two notions)}{\rm
	\begin{description}
\item [In Model Theory:] A subset $X$ of a model $M^n$ is {\em defined} over a set $A$ if there is a formula
$\phi(\xbar,\abar)$ with solution set  $X$.
\item [In usual mathematics] the word `defined' is often short for `well-defined' saying that
    the value of a function defined on a
quotient space does not depend on the choice of a representative.


\end{description}}
\end{remark}

In model theory, we add the adjective `definable' when there is a formula of the language that
captures the notion. Thus, the algebraic geometric `automorphism' becomes `definable bijection'.  It
is worth noting that many
important automorphisms in algebraic geometry do not necessarily preserve structure.

\begin{remark}[Why infinitary logic?]{\rm
A natural question at this point is: Why is axiomatizability in
$L_{\omega_1,\omega}$ relevant to geometric questions?
The answer to this question is not univocal, and strongly reflects different
historical issues arising in different areas of mathematics. We discuss four responses, two from
ordinary mathematics, two from logic.

\begin{enumerate}
\item In ordinary mathematics:
\begin{enumerate}
\item The constraints of expressibility offered by a particular logic force a detailed analysis of
    the hypotheses of a result.  This analysis in similar earlier cases has led to, for example, the
    Zilber-Pink conjecture and the Conjecture on the Intersection of Tori (see
    e.g.~\cite{BreuillardPizarroTentWagner}).
\item Of course, each of the `canonical structures' is explicitly definable in set theory. But this
    definition in most cases
    is useless for studying the object. Useful succinct second order axioms are available
    for the real and complex numbers  but are only partially known for universal covers. First order
    logic is stymied {\em a priori} by the intractability of arithmetic.
   Thus, categoricity in infinitary logic is essential for \textit{giving an
    `algebraic' account of an `analytic object'}. This use of model theory can be seen as
    part of the larger scale \emph{GAGA} mathematical program of bridging analytical
    concepts and algebraic ones.



\end{enumerate}
    \item In logic (in particular, in model theory):
  \begin{enumerate}
\item A natural question is: are there important mathematical notions expressible in infinitary logic which are not
expressible in first order? The study of complex exponentiation yielded a superb initial example: the
categoricity of the covering map of  $\CC^*$  in  [BaysZil].

 \item  This raises the question of what are the {\em new} axioms in this paper that require an
     infinitary description.  The infinite dimension axioms are well known and the switch from
     `standard kernel' to `standard fiber over $z$' (i.e. $q^{-1}(z)$) is unremarkable. It seems the finite index conditions
     (Section~\ref{section:fic}) are not   first order expressible.

\end{enumerate}
\end{enumerate}}
\end{remark}

%


\section{Categoricity, quasiminimality and excellence}\label{cqe}

We give a quick sketch of notions around categoricity\footnote{More specifically, when in model
theory we use the word \emph{categoricity,} we mean categoricity in a specific cardinality or `in
power'. See a thorough discussion of categoricity in various logics  in~\cite[\S
3.1]{Baldwinphilbook} and an exposition of the philosophical import of the notion in~\cite{CrViZi}.}
and the history
of their logical development.




\begin{definition} [Categoricity] \label{catdef}
\begin{enumerate}[1]
    \item   A \emph{theory} $T$ in a logic $\Lscr$ is a collection of $\Lscr$-sentences in a vocabulary $\tau$.
    \item $T$ is \emph{categorical in cardinality $\kappa$} ($\kappa$-categorical) if all models $M$ of $T$
    with $|M|=\kappa$ are isomorphic.
\end{enumerate}
\end{definition}

Although  certain canonical mathematical structures
 are
fruitfully axiomatized in second order logic,  rather than
second order categoricity, we usually consider
these characterizations as defining these structures \emph{in set theory.}  Such
definitions are exactly what it means to be a structure.
Second order categoricity {\em per se} gives no useful mathematical information.
In contrast, $\kappa$-categoricity in first order logic or in
$L_{\omega_1,\omega}$ provides very significant
(combinatorial geometric) information; 
it assigns a dimension to each model.

%
%
%
%
%

\subsection{The Classical Categoricity Theorems}

The following results survey the spectrum of cardinals in which certain types
of theory can be categorical. These theorems are of the form \emph{if a
theory (or a sentence) is categorical in \underline{some high enough}
cardinal(s), then it must be categorical on \underline{a tail of} cardinals.}

\begin{theorem}[Morley's Categoricity Theorem]
   A countable first order theory
   is categorical in one
    uncountable cardinal if and only if it is categorical in all uncountable cardinals.~\cite{Morley65}.
\end{theorem}

\begin{theorem}[Shelah's Categoricity under the weak continuum hypothesis below $\aleph_\omega$]
  Assuming $2^{\aleph_n}< 2^{\aleph_{n+1}}$ a {\em sentence} in
  $L_{\omega_1,\omega}$  that is categorical in  $\aleph_n$ (for every $n<\omega$)
  is categorical in all uncountable cardinals~\cite{Sh87a}, \cite{Sh87b}.
\end{theorem}

\begin{theorem}[Shelah's Categoricity theorem for excellent sentences]
An  {\em  excellent sentence} in
   $L_{\omega_1,\omega}$  is categorical in one
    uncountable cardinal if and only if it is categorical in all uncountable cardinals~\cite{Sh87a}, \cite{Sh87b}.
\end{theorem}

\begin{theorem}[Zilber's Categoricity for quasi-minimal excellent classes]
A quasi-minimal excellent class is categorical in all uncountable cardinals~\cite{Zilberpseudoexp}.
\end{theorem}

\subsection{Pregeometries (matroids) and quasiminimality}

The presence of quasiminimal pregeometries provides an extremely fruitful and natural control of
models in a class (and of their interactions).

 \begin{definition}[Combinatorial Geometry]
A \emph{closure system} is
a set $G$ together with a `closure' relation on subsets of $G$
$$cl: {\cal P}(G) \rightarrow {\cal P}(G)$$
 satisfying the following axioms.

\smallskip

 {\bf A1.} $cl(X)=\bigcup \{ cl (X'): \
X'\subseteq_{fin} X\}$

{\bf A2.} $X\subseteq cl(X)$

{\bf A3.} $cl(cl(X))=cl(X)$

\medskip

$(G,\cl)$ is a {\em pregeometry} if, in addition, we have:

{\bf A4.} If $a\in cl(Xb)$ and $ a \not\in cl(X)$, then $b\in
cl(Xa)$.

\medskip

If points are closed ($cl(\{ a\})=\{ a\}$, for each $a$) the structure is called a \emph{geometry}.

\end{definition}

Pregeometries are virtually the same mathematical objects as matroids.



\begin{definition}\label{mtbas}
\begin{enumerate}
\item
A subset  $D$  of a  $\tau$-structure  M  is \textbf{first order-definable} in  $M$
if there is  $\abar\in M$  and an  $\Lscr_{\omega,\omega}(\tau)$-formula  $\varphi(x,\ybar)$  such
that   $D = \{m\in M: M \models \varphi(m,\abar)\}$. If $\abar \in A\subseteq M$, $D$ is definable with parameters from $A$.

\item $\acl_M(A)$ (the {\em algebraic closure} of $A$ in  $M$) is $\{m\in M
    \colon \phi(m,\overline{a}),\overline{a} \in A \}$, where
    $\phi(x,\overline{a})$ has only finitely many solutions in $M$.
\item $\dcl_M(A)$ (the {\em definable closure} of $A$ in $M$) is defined as was
    the algebraic closure, but replacing `finitely many' by `one'.
\item An infinite definable subset $D$  (or its defining formula
    $\varphi(x)$) is {\em strongly minimal}
if every definable subset of $D$ in every elementary extension of $M$ is finite or cofinite.
\item
 A theory is \emph{strongly minimal} if the
formula $x=x$ is strong minimal.
\end{enumerate}
\end{definition}

The notion of type is a crucial tool in model theory.

\begin{definition}
\begin{enumerate}
\item
The \textbf{first order type} of $a$ over $B$ (in $M$), denoted $\tp_M(a/B)$,  is the set of  $\Lscr_{\omega,\omega}$-formulas with parameters from  $B$  that are satisfied in  $M$  (for  $a, B \subseteq M$).

\item The \textbf{quantifier-free type} of $a$ over $B$ (in $M$), denoted
    $\tp_{\rm qf}(a /B\colon M)$, is the set of \emph{quantifier-free}
    first order formulas $\varphi(x,{\mathbf b})$ such that  $M\models
    \varphi(a,{\mathbf b})$ (as before, $\mathbf b$ ranges over tuples of
    $B$).
\end{enumerate}
\end{definition}

In most contexts, when we just say `the type of $a$ over $B$,' we mean the first order type.
Note also that if a property is defined without parameters in  $M$, then it is uniformly defined in all models of $\th(M)$ (the \textit{theory} of $M$, i.e., the set of all $\tau$ sentences that are true in $M$).

Here are three fundamental observations on strongly minimal sets.

\begin{itemize}
\item A strongly minimal set  admits a combinatorial geometry when the closure is taken as $\acl$ (Definition~\ref{mtbas}).

\item There is a unique type of elements in a strongly minimal set that are not algebraic. This is called the {\em generic type} for $D$.
\item In many important examples (e.g. $DCF_0$), the structure of the model is controlled by its strongly minimal
    sets.
\end{itemize}

Shelah's abstract notion of independence (for some first order theories,
crystallized as \emph{non-forking}) weakens the notion of combinatorial
geometry by dropping {\bf A3}; in some desirable cases this property is
recovered on the points realizing a {\em regular type} and in even better
cases the dimensions of the regular types determine the isomorphism type of
the model. However, {\em a priori,} the existence of  a global dimension is
unusual.




We now look at the generalization of strong minimality, introduced by Zilber, that is central in the connections between
model theory and algebraic geometry described in this paper.

\begin{definition}
[Quasiminimal structure]\label{qmdefstr}
A structure  $M$ is {\em
quasiminimal} if every first
order ($L_{\omega_1,\omega}$) definable subset of $M$ is
countable or cocountable.  Algebraic closure is generalized by saying
$b \in \acl'(X)$ if there is a first order formula with {\bf
countably many} solutions over $X$ which is satisfied by $b$.
\end{definition}


\begin{definition}
[Quasiminimal excellent geometry]\label{qmdef}
 Let $\bK$ be a class of $L$-structures such that $M \in \bK$ admits a
 closure relation $\cl_M$  mapping $X \subseteq M$ to $\cl_M(X) \subseteq M$
  that satisfies the following properties.

\smallskip
\begin{enumerate}
\item \textbf{Basic Conditions}
\begin{enumerate}
\item
 Each $\cl_M$ defines a pregeometry on $M$.
\item For each $X\subseteq M$, $\cl_M(X) \in \bK$.

 \item countable closure property (ccp): If $|X| \leq \aleph_0$ then  $|\cl(X)| \leq \aleph_0$.

\end{enumerate}

%

\item \textbf{Homogeneity}

\begin{enumerate}
\item 
A class $\bK$ of models has {\bf $\aleph_0$-homogeneity over $\emptyset$}  (Definition~\ref{qmdef}) if the  models of $\bK$
are pairwise qf-back and forth equivalent (Definition~\ref{bfdef})


\item
A class $\bK$ of models has \textbf{$\aleph_0$-homogeneity over models}
if for any $G
 \in \bK$ with $G$ empty or a countable member of
$\bK$, any $H,H'$ with $G\leq H, G\leq H'$,  $H$ is qf-back and forth equivalent with $H'$ over $G$.
\end{enumerate}
%

\item $\bK$ is an {\em almost quasiminimal excellent geometry} if the universe of any model $H\in \bK$ is in $\cl(X)$ for any maximal $\cl$-independent set $X \subseteq H$.


\item We call a class   which satisfies these conditions an {\em almost quasiminimal excellent geometry } \cite{BHHKK}.

\end{enumerate}
\end{definition}

An almost quasiminimal excellent geometry with strong submodel taken as $A\leq M$, if $\acl_M(A) =A$,
gives an \emph{abstract elementary class} (AEC)\footnote{See~\cite{Bilgi} for the early history of
the model theory of AECs.}. But
the distinct notion of a quasiminimal AEC (defined in terms of $\leq$ rather than any axioms) is due to \cite{Vaseyqmaec}.


To obtain that the class is complete for $L_{\omega_1,\omega}$, \cite{Kirbyqm, BHHKK} add the requirement of $\aleph_0$-categoricity.

\begin{remark}\label{qmhist} {\rm
	This definition differs only superficially from those in e.g.
\cite{Kirbyqm}, where the connections with the combinatorial geometry was emphasized by
distinguishing the treatment of elements depending on whether they were in $\cl(H)$. {\bf However,}
\cite{BHHKK} required a quasiminimal structure to have a unique generic
type.
%
This requirement fails in the two-sorted treatment we deal with here;  there
 may be $\acl$-bases in each sort. So
 we replace quasiminimality  with \emph{almost quasiminimality} (less explicit in \cite{BHP}) and
  we thus restore Zilber's first intuition (Definition~\ref{qmdefstr}) that quasiminimality means
  that
all definable sets are countable or co-countable.  }
\end{remark}


\begin{remark}[Excellence]\label{exc}{\rm
From Zilber's introduction of the notion in
\cite{Zilberpseudoexp}, it has been known that the axioms
\ref{qmdef}
imply  $\aleph_1$-categoricity. See the exposition in \cite{Baldwincatmon}.
But, without further  `excellence' hypotheses, it was  unknown whether the class had
larger models.
Two  formulations of excellence are 1) \cite{Sh87a,Sh87b}:
$n$-amalgamation of independent systems of models, for all $n<\omega$,
and 2) A local condition on the properties of a `crown' \cite{Kirbyqm}.
Either of these implies the existence of arbitrarily large
models for theories in $L_{\omega_1,\omega}$.
As we discuss in Section~\ref{foex}, influenced by work Hart and Shelah on first order classification theory, the next result (here modified by `almost') clarified the relationship.}


\end{remark}

\begin{remark}{Crucial Fact}\label{cf}
[{Theorem: Bays, Hart, Hyttinen, Kesälä, Kirby}]. Every almost-quasiminimal
class (Definition~\ref{qmdef}) is excellent as described in Remark~\ref{exc}.
Thus, it is categorical in all uncountable cardinalities.
\end{remark}

\section{Modular and Shimura Curves}\label{modshicurves}
\setcounter{section}{4}

We begin with an astronaut's view of the $j$-function and then turn to the
model theoretic treatment of some generalizations.

\subsection{The great confluence}\label{aroundjfunction}

The general form (over a field of characteristic $0$) of an elliptic curve  is
$$y^2 = x^3 + ax + b.$$
At least since Diophantus (3rd century AD), the search for integer solutions for such
equations has been a central question. The cataloguing of such
equations was a major achievement of the 19th century. One key step toward
this classification is to generalize the original problem and look first for
complex solutions. The solution set of an elliptic curve is then a smooth,
projective, algebraic curve of genus one.  It can be thought of as a
`classical torus' $\TT_\tau := \CC/\Lambda_\tau$, where $\tau\in \CC$ and
    $\Lambda_\tau$ is the lattice in $\CC$ (the subgroup of $(\CC,+)$ generated by $\langle 1,\tau\rangle$.

 Klein studied modular and automorphic functions, which provide surprising and deep links
 between geometry, complex analysis and number
 theory. The most famous example  is the $j$-function,
 analytic on $\mathbb{H}=\{z:{\rm im}(z) >0\}$, the upper half plane,
 and maps onto $\CC$  and meromorphic  with some
 poles on the real axis and the following remarkable properties.

 \begin{theorem}[Classification of tori by the $j$-function]\label{Klein}
 The following are equivalent:
  \begin{enumerate}
  \item There exists $s = \left[
      \begin{array}{cc}
        a & b\\
        c& d
      \end{array} \right]\in SL_2(\ZZ)$ such that $s(\tau)=\frac{a\tau +b}{c\tau +d}= \tau'$,
  \item $\TT_\tau \approx \TT_{\tau'}$ in the algebraic geometry sense of Definition~\ref{twonot}.
  \item  $j(\tau)=j(\tau')$
  \end{enumerate}

\end{theorem}

This rather astonishing classical fact paves the way toward modern day
classifications. It provides equivalences between analytic and
number-theoretic notions. Strikingly, $j$ is defined as a rational function
of two analytic functions $g_2$ and $g_3$ (each of them coding so-called
`modularity' properties):
\[ j(\tau) = 12^3\cdot \frac{g_2(\tau)^3}{g_2(\tau)^3-27g_3(\tau)^3}.\]

But where does the word `elliptic' come from?  A meromorphic function is called an {\em elliptic function},  if it is doubly periodic: there are two
$\mathbb {R}$ -linear independent complex numbers
$\omega_1$ and $\omega_2$
 such that $\forall z \in \CC$,
${\displaystyle f(z+\omega _{1})=f(z)}$ and
${\displaystyle f(z+\omega _{2})=f(z)}$. Abel discovered such doubly periodic functions arose from the solutions of elliptic integrals --
originally defined to find the arc length of an ellipse.
Weierstraß used the symbol $\wp$ to denote a family of functions $\wp(z,\Lambda_{\tau})$ where the defining
double sum runs over the elements of the
 lattice $\Lambda_{\tau}$, generated by $1$ and $\tau$.  The crucial property
of the function is that every meromorphic function that is periodic
on $\Lambda_{\tau}$ is a rational combination of $\wp(z,\Lambda_{\tau})$ and  $\wp'(z,\Lambda_{\tau})$. This field of functions is
precisely Abel's field of elliptic functions.

Klein's discovery of the $j$ function unified the results of Weierstraß.
In
his famous investigation of the
psychology of mathematical investigation, Hadamard devotes several pages to Poincaré's
generalization of the $j$-function to the family of functions derived from Fuchsian group actions.
The crucial phrase for us
is \emph{`the transformations I had used to
define the Fuchsian functions were identical with those
of non-Euclidean geometry'} \cite[p 33]{Hadamard}.

This completes a very quick summary of the 19th century predecessors of the theory of moduli
spaces, developed in the next section. This study involves complex analysis, actions by a discrete group,
number theory, and non-Euclidean geometry. The crucial model theoretic step is to
formalize in a  vocabulary for two-sorted structures of the form
\[ {\mathfrak A} = \langle \langle H;\{ g_i\}_{i\in
  \NN}\rangle,\langle F,+,\cdot,0,1\rangle,j:H\to F\rangle\]
where $\langle F,+,\cdot,0,1\rangle$ is an algebraically closed field
of characteristic $0$, $\langle H;\{ g_i\}_{i<\omega}\rangle$ is a set together with countably many
unary function
symbols, and $j\colon H\to F$.

In the next section we provide some of the mathematical background for a formal analysis of these
two-sorted structures.


\subsection{Moduli Spaces } \label{modspsec}

Moduli spaces in geometry are parametrized collections of objects, together
with equivalences that allow us to see when two objects are in some sense
`the same', and with families that articulate the variation between the
objects in the collection. Paraphrasing  the important survey \cite{modsp},
`moduli spaces are a {\em geometric} solution to a geometric classification
problem.' They parametrize
collections of geometric
{\em objects}, they define {\em equivalences} to say when two objects are the `same', and
establish {\em families} that determine how we allow our objects to vary or modulate.


In model theory, the notion of a uniform family of definable sets has been thoroughly studied.
Such a family is given by a formula of the form $\phi(\xbar,\ybar)$.
Each set in the family is the solution set of $\phi(\abar,\ybar)$ (for some $\abar$), and the set
$\{\abar:(\exists \ybar) \phi(\abar,\ybar)\}$ is an indexing set of the family.
In the algebraic geometry setting, one can require that
 the $\xbar$ fall into a variety  $V$  and the $\ybar$ into a variety $W_{\abar}$. $V$ is a step
 toward the notion of a {\em moduli space}.

Except in \S~\ref{section:notop}, we consider  moduli spaces arising from a
pair $(G,X)$ consisting of a
group $G$ acting on a space $X$. 
The algebraic varieties we study arise as  quotients $\Gamma \setminus X$
(for $\Gamma$ a subgroup of $G$, see Definition~\ref{quotdef}).  A modular
curve arises as a connected component of  quotient of $\HH$ by congruence
subgroups (Definition~\ref{disgrpterm}) of $\gl_2({\mathbb{R}})$.  Shimura
generalized the topic to groups acting on wider classes of domains.   Shimura
curves are rather more complicated yet generally share similar categoricity
properties. Shimura varieties of higher dimension raise many new issues that
we sketch in Section~\ref{Shivar}. In this section, we consider only covers
of modular curves by $\HH$.

Here, $\mathbb{H}$ refers, as in the rest of this paper, to the upper half
complex plane (as the set of points: $\mathbb{H}=\{z\in \mathbb{C}:{\rm
Im}(z)>0\} $). $\mathbb{H}$ is  also called the hyperbolic plane (when
endowed with a metric and topology that make it hyperbolic rather than
Euclidean). See~\cite{Miyake} for a detailed description. In all our
examples, the function $p$ maps the hyperbolic plane into   a complex variety.
We consider the action of  $\psl_2({\mathbb{R}})$ on $\HH$ as fractional linear
 transformations:
for\\ $A =\left[
      \begin{array}{cc}
        a & b\\
        c& d
      \end{array} \right]\in \sl_2(\ZZ)$
and $\tau \in \HH$,
$A(\tau) =\left( \frac{a\tau +b}{c\tau +d}\right) $.

The group of bijections  (isometries, $\isom(\HH)$) that preserve the hyperbolic metric of $\HH$ is generated by
$\psl_2({\mathbb{R}})$ and the map $z\mapsto -\overline{z}$; $\psl_2({\mathbb{R}})$  consists  precisely of all those isometries that preserve orientation (e.g. \cite{Katok}).
After outlining here the classical theory of such actions and moduli spaces,
in section \ref{section:QEcurves} we describe
 a model theoretic approach.

%
%
%
%
%
%

\begin{definition}[Fuchsian group]\leavevmode \label{fuchgrp}

\begin{enumerate}
    \item A subgroup $G\le \isom(\HH)\approx \psl_2({\mathbb{R}})$ is
        \emph{discrete} if it is discrete in the induced
	topology.
    \item A {\em Fuchsian group} is a discrete subgroup of $\psl_2({\mathbb{R}})$.
\end{enumerate}
\end{definition}
The most important example of a Fuchsian group is $\psl_2(\ZZ)$. Underlying this entire study and
almost one and a half centuries of interactions between number theory and complex analysis
is the remarkable fact that the quotient of $\HH$ by certain  discrete subgroups has the structure
of a Riemann surface \cite[\S 1.8]{Miyake} and even an algebraic
variety which, in important cases, is a moduli space \cite{milnenams}.
%

\begin{definition}[Quotient of $\HH$ by a group]\label{quotdef}
If a group $G$ acts on a set $X$, $G\setminus X$ has universe the collection
of $G$-orbits of the action. $\pi$ is the canonical map taking $x$ to its
orbit $Gx$. The prototypical example corresponds to $X=\mathbb{H}$.
\end{definition}

\begin{definition} The quotients $V = S(\CC)$ of $\HH$ by a discrete group $\Gamma$ that we
consider are examples of {\em moduli spaces}. $V = \bigcup_{a\in C}V_a$    is
the image of a map $p$ from $\HH$ that acts as a {\em uniformizer} for a
family of varieties $V_a$. Namely
for each $a,b \in \HH$, $V_{a} \cong V_{b} $ iff for  some $\gamma \in \Gamma $, $\gamma(a) =b$ iff $p(a) = p(b)$.
\end{definition}

We explored in
Section~\ref{aroundjfunction} the ur-example of a moduli space, elliptic curves as uniformized by
the $j$-function.
The next definition relies on the fact that, while elements of $\psl_2({\mathbb{R}})$ fix $\HH$
setwise, they also act on all of $\CC$.




\begin{definition}[Cusp]\label{cuspdef} For a discrete subgroup $\Gamma$
of $\psl_2({\mathbb{R}})$:
\begin{enumerate}
\item 
We say $c \in {\mathbb{R}} \cup \{\infty\}$ is a {\em cusp} of $\Gamma$ if
    $c$ is the unique fixed point of some $\gamma \in \Gamma$.
\item $P_\Gamma$ is the set of cusps of  $\Gamma$
and  $\HH^* = \HH^*_\Gamma  = \HH \cup P_\Gamma$.
\end{enumerate}

\end{definition}


We relate some standard facts (see \cite[p 15]{Harristhe}). The first relies on
the fact that while some of the quotients we study are not compact,
they can be compactified by adding finitely many cusps from ${\mathbb{R}} \cup \{\infty\}$.


\begin{fact}\label{quotchar} For any discrete
subgroup $\Gamma \subseteq  \psl_2({\mathbb{R}})$, the quotient $\Gamma
\setminus \HH^*_\Gamma$ is  a compact Hausdorff space that can be given the
structure of a Riemann surface. Therefore if $\Gamma'$ is of finite index in
$\Gamma$, the quotient $\Gamma' \setminus \HH^*_\Gamma$  is a compact Riemann
surface, and is therefore algebraic by the Riemann existence theorem.
 $\HH^*_\Gamma$ is the compactification of the quasi-projective
algebraic variety (so first order definable) $\HH_\Gamma$, .

\end{fact}

For the purposes of this paper since the quasiprojective variety $\HH_\Gamma
=\Gamma \setminus \HH$ determines the (classical) algebraic variety (set of
solutions of a system of polynomial equations),  $\HH^*_\Gamma$ we work
hereafter with $\HH_\Gamma$. This is natural from a model-theoretic
standpoint since (in this situation) there are only finitely many cusps and
so the sets differ by only finitely many points.

Notation~\ref{fixG} {\em fixes the group $G$  for the rest of
\S~\ref{modshicurves}.} Setting the determinant as $1$ and modding out the
center guarantees the group action preserves both distance and orientation.

\begin{notation}\label{fixG} Let  $G=\gl_2^{ad}(\QQ)^+
 =_{\rm def} PSL_2({\mathbb{Q}})/Z(PSL_2({\mathbb{Q}}))
 \approx PSL_2(\QQ)$ modulo its center. $\Gamma$ varies over subgroups of $G$
 \end{notation}

We now distinguish two kinds of points in $\mathbb{H}$: `special' points and
`Hodge-generic' points. The equivalence of the following definition with the
usual notion \cite[Definition 2.2]{DawHarris} for Shimura varieties is in
\cite[Theorem 2.3]{DawHarris}.

\begin{definition} [Special points]\label{spptdef}
Fix $\langle \HH,S(\CC),p\rangle$ with $S(\CC)$  biholomophic to $\Gamma
\setminus \HH$. A point $x\in \HH$ is {\em special} if there is a $g\in G$
whose unique fixed point  is $x$.
\end{definition}

We omit the definition of a Hodge generic point arising in algebra, as it
does not enter our discussion; we use only the equivalent characterization
\cite[Prop 2.5]  {DawHarris} given in Fact~\ref{dichot}.1) and the dichotomy
in 2) noted just after that proposition. It is worth mentioning that for a
point the fact of being ``special'' or ``Hodge generic'' does not depend on
the choice of the group $\Gamma$; furthermore, these two notions are
preserved by the action of $G=GL^{\rm ad}_2(\mathbb{Q})^+$.

\begin{fact}\label{dichot} Special and Hodge generic  points \cite[Proposition 2.5]{DawHarris}
\begin{enumerate}[(1)]

\item If $x$ is Hodge generic the only $g \in G$ that fixes $x$ is the
    identity.

\item  Every point in $\HH$ is either Hodge generic or special.
\end{enumerate}
\end{fact}



Although we are studying the categoricity of the universal cover of a specific modular curve (e.g.
the image of the $j$-function, $\Gamma \setminus \HH$), other modular curves naturally arise in the
analysis.
The study of families of such curves is expounded in \cite[\S 6, 7]{Shimura}.
A key tool to give a uniform treatment to a family is the existence of a common commensurator of
the generating Fuchsian groups. In fact, the members of the family are interalgebraic and  the entire family (indexed by the $\Gamma_N$) is studied in   \cite{DawZil1}.

\begin{definition}\label{disgrpterm} \begin{enumerate}
%

    \item The groups $\Gamma_N$ ($N$ a fixed integer) are given by
    $$\Gamma_N = \left\{\left[
      \begin{array}{cc}
        a & b\\
        c& d
      \end{array} \right] \in \Gamma : \  b \equiv c \equiv 0, a \equiv d \equiv 1 \  {\rm mod}\ N \right\}.$$

Note that each $\Gamma_N$ has finite index in $\Gamma$ and if $N|M$ then $\Gamma_M \subseteq \Gamma_N$.

    \item Two subgroups  $\Gamma$ and
	$\Gamma'$ of a group $H$ are said to be \emph{commensurable} if
$\Gamma \cap \Gamma'$
is of finite index in both of them.

\item A {\em congruence subgroup} is a  subgroup $\Gamma'$ of $\Gamma$ such
    that
    some $\Gamma_N$ is a finite index subgroup of $\Gamma'$.

\item  The \emph{commensurator} ${\rm comm}(\Gamma)$ of a subgroup $\Gamma$ of ${\psl}_2(\RR)$
is  $$\{\delta \in  {\psl}_2(\RR)\colon \delta \Gamma \delta^{-1}\
\text{\  is  commensurable with}\
 \Gamma \}.$$

\end{enumerate}
\end{definition}

We rely on the following standard fact.

\begin{lemma} The group $G =\gl_2^{ad}(\QQ)^+$  (Notation~\ref{fixG}) is the
 commensurator of any congruence subgroup $\Gamma$ of $SL_2(Z)$.
\end{lemma}

Because the functions $g\in G$ are in the formal vocabulary, we employ
congruence subgroups $\Gamma_\gbar$ from Notation~\ref{zgbar} rather than the
$\Gamma_N$. The $Z_{\gbar}$ defined in Notation~\ref{zgbar} play a central
role both
 in the quantifier elimination and via an inverse limit in Section~\ref{section:fic}.

\begin{notation}\label{zgbar} With $G$ as fixed in Notation~\ref{fixG},
as  each of the congruence subgroups of $\psl_2(\ZZ)$  act on $\HH$ we can define
for any finite
sequence of the form $\gbar = \langle e,g_2, \ldots, g_n\rangle$  from $G$
 (by convention, $g_1=e$),
\begin{enumerate}
\item $\Gamma_{\gbar}= \Gamma \cap g_2^{-1}\Gamma g_2 \ldots
 \cap g_n^{-1}\Gamma g_n$.
\item Let $p: \HH \rightarrow S(\CC)$.
\begin{enumerate}
\item 
    $Z_{\gbar}$ is defined as $\{\left(
    p( x),p(g_2 x),\ldots,p(g_n x) \right) \in S(\CC)^n:x\in \mathbb{H}\} $.



    \item
Let $p_{\gbar}:\HH  \rightarrow Z_{\gbar} \subseteq S(
\CC)^n$ be defined by
$$x \mapsto p(\gbar x) = \langle p(x),p(g_1(x)) , \ldots, p(g_n(x))\rangle .$$
\item Let $[\phi_{\gbar}]$ be the map from $\HH_{\gbar}$ onto $Z_{\gbar}$
    (Lemma~\ref{holmo}) given by $[\phi_{\gbar}]{x}_{\Gamma_\gbar} =
    p_\gbar(x)$.
\end{enumerate}
\item 
%
%
Let $\HH_\gbar$ denote
$\Gamma_{\gbar}\setminus \HH$.
\end{enumerate}
\end{notation}

All the previous are well-defined by our choice of $p$ and $\Gamma$.


The following  lemma \cite[3.31]{Etrev} is central to Section~\ref{galrepsec}. Its proof uses
Shimura theory very heavily.

\begin{lemma}\label{holmo}

The map $[\phi_\gbar]$ is  bijective on the Hodge generic points and the
image $Z_\gbar$ is a {\bf variety} contained in $\Sscr^{n}(\CC)$, $n=
\lg(\gbar)$. Moreover, \cite[p 17]{Etrev}, for all $\gbar$, $Z_\gbar$ is
defined over the maximal Abelian extension $L$ of the field of definition,
$E$,  of $S$.
\end{lemma}

\begin{remark}\label{defconf} {\rm From the model theoretic standpoint, it makes
 no sense to say the $[\phi_\gbar]$ are definable since their domains $\HH_\gbar$ are not.
 While the maps $[\phi_\gbar]$ are bijective on Hodge generic points, they may
 identify special points.}
\end{remark}

\subsection{Quantifier Elimination in Modular and Shimura Curves }\label{section:QEcurves}

We now lay out the vocabulary and first order theory for studying modular
curves. The mathematical input is a Fuchsian group $\Gamma$ acting on
hyperbolic space $\HH$ and the image curve $S(\CC) = \Gamma \setminus
\HH^*_\Gamma$ (Definition~\ref{cuspdef}) with a standard model $\pbar =
\langle \HH,S,p\rangle$. The structure of a discrete group  is unwieldy from
a traditional model theoretic standpoint because its first order theory is unstable and
undecidable. Just as modules are usually studied in model theory by adding
unary function symbols $f_r$ for the elements of the ring, in order to
represent the action of $G$ on $\HH$, we add symbols $f_g$  for $g\in G$ as
unary functions that act on $\HH$. We thus use a two-sorted presentation of
our structures: a sort for the domain, a sort for the target, and a map $p$ connecting them.
%


\begin{remark}[Sorts]\label{sortexp} {\rm A two-sorted structure interprets two
sort symbols and additional relation and function symbols with the
understanding that each such relation/function either is restricted to one of
the predicates or explicitly connects them.}
\end{remark}

\begin{notation} [The formal vocabulary $\tau$]\label{mcvocab}
 {\rm The two-sorted vocabulary $\tau$ consists of the sorts (unary predicate symbols)  $D$ (the
     covering sort), $S$ the target sort, and  a  function $q$ mapping $D$ onto the sort $S$.


We write $\tau_G$ for the vocabulary of the first sort with $G =G^{ad}(\QQ^+)$. The second $\tau_F =
\Rscr $ where  $\Rscr$  is the set of formulas in  $\{+,-,0,1,\times\}$ specified in
Definition~\ref{stmod}.  $\tau$ is $\tau_G \cup \tau_F \cup \{p\}$.
There are {\em constant symbols} for each element of the field $E^{ab}(\Sigma)$ defined in
Notation~\ref{stmod}.  We use $f_g$ to name the functions
acting on $D$, but often write the shorter $g(x)$ or $gx$ instead of $f_g(x)$.}
\end{notation}




The following notation is essential to understand the Axioms~\ref{mcax}. Note
in the prototype $q$ is replaced by the known covering map $p$.

\begin{notation}\label{stmod} {\rm  The {\em standard model} for a
{\em modular curve} determined by a Fuchsian group
 $\Gamma \subseteq     G = G^{ad}(\QQ^+)$ will consist of a $\tau$-structure
  $\pbar =\langle \HH, S, p\rangle$ with
the domain $\HH$,   the variety $S(F)$ over the algebraically closed field
$F$ defined by $\Gamma \setminus \HH $, and $\Rscr$ the set of
all Zariski closed relations on $S(F)^n$ (for all $n$) with constants from a
field $E^{ab}(\Sigma)$ that are true in $F$. $E^{ab}$ is  the  maximal abelian extension of the
defining (reflex) field $E$ of $S$.  $E^{ab}(\Sigma)$ is the extension of $E^{ab}$ ($F_0$ in
\cite[p 19]{Etrev})  obtained by adding  the
coordinates of the ($\leq \aleph_0$) special points,
 and closing to a field.}

 \end{notation}

 \begin{notation}\label{thnot} {\rm For a structure $\pbar$, we write  $\th(\pbar)$ for the complete
 first order theory of all sentences true in $\pbar$ and $T(\pbar)$ for the specified set of axioms
 true of $\pbar$. Clearly, $T(\pbar)\subseteq Th(\pbar)$.}
 \end{notation}

We must distinguish  $\th(\pbar)$  from its subset $T(\pbar)$ until
we prove $T(\pbar)$ is a complete axiomatization of $\th(\pbar)$.



\begin{definition}[First Order Axioms]\label{mcax}
    $T(\pbar)$ is the following  collection of first order sentences that are to hold
in a structure $\langle D,S(F),q\rangle$.
\begin{enumerate}
\item Each sentence in $\th(\langle\HH, \{f_g:g \in G\rangle)$. These
    include `Special Point axioms' $SP_g$:
For each  $g\in G$ that fixes a unique point in $D$

$$\forall x, y \in D
[(g(x) =x \wedge  g(y) =y) \Rightarrow x=y] $$
\item  $\th(S(\CC),\Rscr)$ ($\Rscr$ from Definition~\ref{mcvocab}.)

\item The covering map; for each $\gbar \in G^m$ and all
$m<\omega$:

\begin{enumerate}
\item $Mod^1_{\gbar}$,
$$\forall x \in D \ (q(g_1(x), \ldots q(g_m(x)) \in Z_{\gbar})   $$
\item  $Mod^2_{\gbar}$: $$\forall z \in Z_{\gbar} \exists x  \in D
(q(g_1(x)), \ldots q(g_m(x)) =z ) $$
\item $${\bf MOD}  = \{ Mod^1_{\overline{g}}\wedge Mod^2_{\overline{g}}: \overline{g}\in G^{m},
    m<\omega\}$$
\end{enumerate}
%

\end{enumerate}

\end{definition}
Note that $\bf{MOD}$ is a countable collection of first order sentences.



\begin{notation}\label{namesp} {\rm By the choice of $E^{ab}(\Sigma)$, special points belong to $\dcl\left( \emptyset \right)
	$. Therefore,
	we can name each one of them by $d_g$, where $g \in G$ fixes $d_g$.
 Any $g$ that fixes a point is in $G -\sl_2(\ZZ)$ \cite[Lemma 3.18]{Etrev}.  There will be
 distinct $g_1,g_2$ that fix the same point (e.g. if $g_2 = g_1^2$). If so,    $T(\pbar) \vdash d_{g_1}
 = d_{g_2}$
 The theory of $(D,G)$ contains the uniqueness axiom
 (Definition~\ref{mcax}.1) that entails $g(d_g) = d_g$.}


\end{notation}
The cover sort is  a set with unary functions. Both its theory (since the
universe is a union of orbits) and that of the field sort (since
algebraically closed) are strongly minimal  and quantifier eliminable.


\begin{definition}\label{bfdef} {\rm  We say two structures $M$ and $N$ are {\bf qf}-back and forth
    equivalent if the system $I$ of  partial isomorphisms of $M$ and $N$ between isomorphic {\em
finitely generated substructures} satisfies the back and forth condition: For
each $f \in I$ and each $m \in M - \dom f$, there exists an $n \in N$ such
that $f\cup \{\langle m,n\rangle \} \in I$, and symmetrically, for each $n\in N-\im\, f$,
there exists $m\in M$ such that $f\cup \{\langle m,n \rangle \} \in I$. In this situation $\dom
f$ is
 definably close.} 
\end{definition}

\begin{notation}\label{barnot}{\rm We  write $\gbar(x)$ for $(g_1 (x), \ldots g_n (x))$ where
	$\gbar$ has length $n$ and begins with $e$. And then $\gbar(\xbar)$
denotes the sequence of length $nm$ obtained when $\gbar$ is applied to each
element of a sequence $\xbar \in  (D)^m$.  When convenient we write $gx$ or
$\gbar x$ for the action, omitting the parentheses.}
\end{notation}


We now sketch the proof of Theorem~\ref{qe} that $T(\pbar)$ axiomatizes a
complete, quantifier eliminable $\tau$-theory.

\begin{definition}[The back and forth]\label{ps} {\rm
Fix two models $\qbar =\langle D,S(F),q\rangle$ and $\qbar' =\langle D',S(F'),q'\rangle$
 of $T(\pbar)$. We define the {\bf
qf}-back-and-forth system  $I$ of substructures  of  $\qbar$ and $\qbar'$ For
each $f \in I$, $\dom f$ and ${\rg \  f}$  are each finitely generated over
$E^{ab}(\Sigma)$ .
 A typical member $f$ of the system for $\qbar$ has
 $\dom f = U = U_D \cup
U_S$. Since $U$ is finitely generated, $U_D$ consists of the $G$-orbits of a
finite number of $x\in D$; $U_S$ is $S(L_U)$
 where $L_U$
is the field generated by $E^{ab}(\Sigma)$ (since the elements of
$E^{ab}(\Sigma)$ elements are named), the coordinates of the $q(x)$ for $x
\in U_D$ and finitely many additional points of $F\cap U$.  Note that the
additional points determine finitely many new field elements since $q$ is
constant on each orbit, so the field remains finitely
 generated. Define a
similar subsystem for $\qbar'$,  labeling by putting  primes on corresponding objects.
By Lemma~\ref{dichot} every point of $D$ is either special and so named in the vocabulary
(Remark~\ref{namesp}), or Hodge generic. Thus we can ignore the special points in building the back
and forth system.
}
\end{definition}


Suppose $f$ is an isomorphism between $U \subseteq \qbar$ and $U' \subseteq
\qbar'$. Then $f$ restricts to a $G$-equivariant (elements in the same orbit
have the same image) injection of $U_D$ into $U_{D'}$ and an
 embedding of $S(L_U)$ into $S(F')$ induced by an embedding   $\sigma$ of $L$ into $S(F')$,
 that fixes $E^{ab}(\Sigma)$.

%



Note that the following claim is for arbitrary finite sequences $\gbar$, but only singleton $x$.
The type $r_d$ of an infinite sequence (here represented by an infinite tuple of variables $\vbar$)
includes the types of $\gbar x$ for any finite $\gbar$.

The main consequence of the following claim is that we may reduce types of points in the domains
sort to quantifier-free types of their images in the field sort.

\begin{claim}\cite[Prop 3.3]{DawHarris}\label{3.3} If $d \in D-U_D$ is
Hodge generic:
$$r_d(\vbar) \models \tp_{qf}(d/U),$$
where $r_d(\vbar) =\bigcup_{\gbar \in G} \tp_{qf}(q(\gbar(d)) /U) =
\tp_{qf}(\langle q(g  (d)): g\in G\rangle /U)$.
\end{claim}



\begin{proof}
We show that there is a unique quantifier-free type over $U$  of an element
of $D$  that restricts to $r_d$. The consistent non-trivial types in $\tau_G$
are i) $\{x \neq f: f\in U_D\}$ and ii) $\{x \neq g x\}$ for any non-identity
$g  \in G $. The first is captured by  $(q(x),q(f)) \not\in Z_{e,e}$ for each
$f \in U_D$  and the second by
$(q(x),q(x)) \not \in Z_{e,g}$ if $g \not \in \Gamma$ and these are both in $r(\vbar)$. 

Suppose $\hbar \in S(M)^{\omega}$ (for a saturated $M \models T({\mathbf p})$
containing $U$) realizes $r_d(\vbar)$ and $\hbar$ with $d'\in D(M)$ satisfy
$\hbar = \langle q(g (d')): g \in G\rangle$. By the previous paragraph $d'
\not \in U_D$. So $d'$ realizes $\tp_{qf}(d/U)$ as required.
\end{proof}






%

\begin{notation}\label{ftype}  For a type $r(v)$ over a set $A$ and an isomorphism $f$ from $A$ to $B$, $f(r)$
is the set of $B$-formulas $\phi(v,f(\abar))$ with $\phi(v,\abar)\in r$.
\end{notation}


\begin{claim}\cite[Prop 3.4]{DawHarris}\label{onegbar}  Fix $\gbar$. If $x \in U_D$,
 there is   an $x'\in U_{D'}$ such that 
 $q(\gbar (x'))
  \in S(F')^m$ realizes
$f(\tp_{qf} (q(\gbar (x)) /L_U))$.
\end{claim}

\begin{proof} We write $Z_{\gbar}^{\qbar}$ for the points in $S(F)$ satisfying (the formula defining) $Z_{\gbar}$.
  Using Notation~\ref{ftype}, Claim~\ref{3.3} implies  that the smallest algebraic subvariety  $W_{\gbar}^{\qbar}$ of $S(F)^n$ that is defined over $L_U$ and contains
 $q(\gbar(\xbar))\in S(F)^n$ determines $\tp_{qf}(\gbar(\xbar)) /L_U)$. Since $Mod^1_{\gbar}$ is true in $\qbar$, $W_{\gbar}^{\qbar} \subseteq Z_{\gbar}^{\qbar} $.
 But since (by Lemma~\ref{holmo}) $Z_{\gbar}^{\qbar}$ is fixed setwise by $\sigma$ (the map
 described after Definition~\ref{ps}) - being defined over $E^{ab}(\Sigma)$, we have that
 $Z_{\gbar}^{\qbar'}=Z_{\gbar}^{\qbar}$, and therefore
 $W_{\gbar}^{\qbar'} \subseteq Z_{\gbar}^{\qbar'} $. Now applying
$Mod^2_{\gbar}$ in $\qbar'$, we find the required $x'$.
\end{proof}

Having proved  Claim~\ref{onegbar}, we can finish the argument.  We need one more crucial piece for the
`forth'. What if $x \in D- U_D$? For this, we need   $\qbar'$ to be $\omega$-saturated (realize all types over finite sets).

\begin{theorem}\label{qe} Suppose  that $\qbar$ and $\qbar'$ are $\omega$-saturated. Then the
    $qf$-system described in Remark~\ref{ps} is
a back and forth; hence, $T(\pbar)$ admits elimination of quantifiers and is complete.
\end{theorem}
\begin{proof} 
Suppose $f$ is an isomorphism between $U \subseteq \qbar$ and $U' \subseteq \qbar'$. Then $f$ restricts to  a $G$-equivariant injection of $U_D$ into $U_{D'}$ and an embedding of $S(L_U)$ into $S(F')$ induced by an embedding   $\sigma$ of $L_u$ into $S(F')$, that fixes $E^{ab}(\Sigma)$.

 For $x \in \qbar - U$, we must find $x' \in U'$ so that $f \cup (x,x')$ generates
an isomorphism between the structures generated by $U \cup \{x\}$ and
$U' \cup \{x'\}$.
If $x\in S$, $x = q(\check x)$ for some $\check x\in D$ so we restrict to that case.
If $x \in U_D$,
 $x'$ exists as $U'_{D}$ is closed under
action by $G$.  Since the coordinates of special points are in
$E^{ab}(\Sigma)$, whose points are all named, for a special point $x$, $x'$ must equal $x$.

The difficult case is when $x \in (D- U_D)$ is Hodge generic. But we noted in
Claim~\ref{3.3} that it suffices to simultaneously realize  all types
$\tp_{qf}((q(g_1 x), \ldots q(g_n x)) /U)$ for all $\gbar$ (of arbitrary
length). A slight variant on the argument for Claim~\ref{onegbar} still holds
if for fixed $x$, we replace a single $\gbar$ by an arbitrary finite set of
$\gbar$. By compactness, the entire type is consistent and so satisfied in
the $\omega$-saturated $\qbar'$.  There is one final step.  By induction we
have to choose $x'$ for a sequence $\xbar,\ybar,x$ where $\xbar \in U_D$ and
$\ybar \in U_S^k$ for some $k$. But what if $x\in U_S$? By Claim~\ref{3.3},
$\tp_{qf}(\xbar,\ybar) $ is determined by $\tp_{qf}(\gbar(\xbar),\ybar) $ (in
the field sort). That we can  choose of $x'\in U'_S$ to satisfy
$f(\tp_{qf}(\gbar(\xbar),\ybar)) $ is now immediate by $\omega$-saturation
and quantifier elimination in the field-sort.


By Karp's theorem \cite[Theorem 3]{Barwisebf}, the existence of the back and
forth implies all $\omega$-saturated models of $T(\pbar)$ are
$L_{\omega_1,\omega}$ (indeed, $L_{\infty,\omega}$) elementarily equivalent.
Every model has an $\omega$-saturated elementary extension, so $T(\pbar)$
is complete.
\end{proof}

\subsection{Galois Representations and finite index conditions}\label{section:fic}

In this section we begin by considering the action of discrete and Galois
groups on the domain and field sorts. Then we unite these approaches by
defining a Galois representation. We  then state the key to establishing
categoricity, a consequence of Serre's open mapping theorem.

\subsubsection{Two views: domain and field sort}\label{domfld}
We explore the following diagram which links the domain sort (via the
quotient) with the field sort.

\[ \begin{tikzcd}
    \mathbb{H}_{\overline{h}}\approx \Gamma_{\overline{h}}\setminus \mathbb{H}
    \arrow{r}{[\phi_{\overline{h}}]} \arrow[swap]{d}{{\rm id}_{\mathbb{H}_{\overline{g}}}} &
    Z_{\overline{h}} \arrow{d}{\psi_{\overline{h},\overline{g}}} \\%
    \mathbb{H}_{\overline{g}}\approx \Gamma_{\overline{g}}\setminus \mathbb{H}
    \arrow{r}{[\phi_{\overline{g}}]}& Z_{\overline{g}}
\end{tikzcd}
\]

\begin{convention} \label{conv1} $\gbar = \langle e, g_1 \ldots g_{n-1}\rangle$ has length $n$.
We restrict
to $\gbar$ with $\Gamma_{\gbar}\unlhd \Gamma$ (normal subgroup). Recall
$Z_{\gbar} \subseteq S(\CC)^{\lg(\gbar)}$.
\end{convention}

  We have two views of `essentially' the same map. The first moves to a quotient
   on the domain side which is not $\tau$-definable;
   the second
  `names' the range of the first in the target side.
  We begin with {\em quotient data} but with
manifestations in both the domain and target.

\medskip
\noindent
\textbf{Domain/Quotient data:} The first view motivates $id$ for identity.

\begin{definition}\label{idmap}
Let $\gbar \subseteq \hbar$.
Define $\id_{\hbar\gbar}: \HH_{\hbar} \rightarrow \HH_{\gbar}$
 by
 $[x]_{\Gamma_\hbar}\mapsto [x]_{\Gamma_\gbar}$.
\end{definition}

 The normality hypothesis implies that $\Gamma_\gbar/\Gamma_{\hbar}$ acts
on $\HH_{\gbar}$: for $\lambda\in \Gamma_\gbar$,   $\lambda[
x]_{\Gamma_\gbar} := [\lambda x]_{\Gamma_\gbar}$, so the representatives
$\lambda_i$ of the cosets of $\Gamma_{\gbar}/\Gamma_{\hbar}$ index the
equivalence classes; thus the action is transitive.
%
%
%

%
%
%
%
%

%

%
%
%



\medskip
\noindent
\textbf{Field data:} We define the right hand column of the diagram.

%
%

\begin{definition}\label{fieldmaps}
\begin{enumerate}
\item For $\gbar \subseteq \hbar$,  $\lg(\gbar) = n$, $\lg(\hbar) = m$,
    $\psi_{\hbar,\gbar}$, denotes the restriction of the natural projection
    from $S(\CC)^m$ onto $S(\CC)^m$ to  a map from  $Z_{\hbar} \subseteq S(\CC)^m$
    onto $Z_{\gbar} \subseteq S(\CC)^n$.
\item Choose $z \in Z_\gbar$ and let $L= L_z$ be a finitely generated
    extension of the defining field for $S$ such that $z$ is defined over
    $L$. Write $\overline{L}$ for $\acl(L)$.

    \item  Now, $\aut(\CC/L)$ acts on the fiber of $\psi_{\hbar,\gbar}$
        over $z$, by
its action on the coordinates of $z$;  as it would for any definable
finite-to-one map from $Z^m_{\hbar} \rightarrow Z^n_{\gbar}$.

    \end{enumerate}

    \end{definition}

To connect the two sides, conjugating by $[\phi_\hbar]$,
$\aut(\overline{L}/L)$ acts on $\id_{\hbar \gbar}^{-1}(z)$.

\begin{lemma}\cite[\S 3.5 p. 18]{Etrev} \label{varside}
$\aut(\CC/L)$ acts on the fiber of $\psi_{\hbar,\gbar}$ over $z$,
 (and so via $[\phi_\hbar]$ on  $\id_{\hbar \gbar}^{-1}(z)$).
 This action commutes with the action of the free and transitive (simply transitive)
  action of $\Gamma_\gbar/\Gamma_\hbar$ on the fibers of $\id_{\hbar,\gbar}$.
Thus we have a homomorphism (Galois representation) $\rho^z_{\gbar,\hbar}$
from $\aut(\overline{L}/L)$ into $\Gamma_\gbar/\Gamma_\hbar$.

\end{lemma}
%


\subsubsection{Galois Representation}\label{galrepsec}

While the notion of a  representation  of a group $A$ frequently refers to
linear representations,  a homomorphism of $A$ into a matrix group $B$, here
we will discuss specific examples of a more general notion: {\em a
representation of $A$ is a homomorphism of $A$ into a group $B$. This is a
{\em Galois representation} if $A$ is the Galois group of one field over
another.}  In Section~\ref{domfld}, we gave Galois representations of
$\aut(\overline{L}/L)$ into $\Gamma_\gbar/\Gamma_\hbar$. In order to
understand how to combine the actions of the $\Gamma_\gbar/\Gamma_\hbar$ as
$\gbar, \hbar$ vary, we need the notion of inverse limit.


\begin{definition}[Inverse Limit]\label{thelimit} Given a directed set $(I,\leq)$ an
{\em inverse system} on $I$ is a family of structures $\langle A_i:i\in
I\rangle$, and for $i<j$, maps $f_{ij}$ from $A_j$ to $A_i$ such that
$i<j<k$ implies   $f_{ij} \circ f_{jk} =  f_{ik} $.

An {\em inverse limit} of this inverse system is an object $\hat A= \underleftarrow{lim} A_i$ and a
family of morphisms
$g_i: \hat A \rightarrow A_i$ such that (1) for all $i<j$ in $I$, $f_{ij} \circ g_j = g_i$ and (2)
given any $A'$ and family ${g'}_i$ satisfying (1) there is a unique morphism $h:\hat A\rightarrow
A'$ such that for all $i\in I, {g'}_i=g_i\circ h$.

\end{definition}

 \begin{definition} {\bf  Galois Representations of Inverse Limits} \label{galrep}
 We work with a  modular curve $S(\CC) = \Gamma \setminus  \HH$
 which is defined over $E^{ab}(\Sigma)$.
  (Notation~\ref{stmod}).
Since each  $\Gamma\gbar \subseteq \Gamma$,
$\rho^z_{\gbar,\hbar}:\aut(\overline{L}/L) \rightarrow \Gamma$ and by taking
an
  inverse limit of the representations  $\rho^z_{\gbar,\hbar}$, 
we obtain:


 $$\rho^z \colon
{\rm Gal}(\overline L/L) \rightarrow \overline \Gamma$$ where $\overline
{\Gamma} =   \underleftarrow{\lim}_\hbar\ \Gamma /  \Gamma_{\hbar}$. The
$\hbar$ range over all finite sequences as Convention~\ref{conv1}.  See
Definition~\ref{thelimit} and \cite[\S 3.6 p 17]{Etrev}.
\end{definition}



For any groups $H_1\leq H_2$ that act on a set $X$ the $H_1$-orbits of $X$
partition the $H_2$-orbits.  So if $[H_2:H_1]$ is finite and $H_2$ is
infinite, the obits will have the same cardinality and the smaller
$[H_2:H_1]$ is, the closer we are to an isomorphism.

%
%
%

%
%
%


%

Now, we can state the first of two crucial sufficient conditions for
categoricity.

\begin{definition} {\bf First Finite Index Condition (FIC1)} \label{fic1}
The {\em first finite index condition} is satisfied by a modular curve
$p\colon \HH \rightarrow S(\CC)$ if:

For any non-special points $x_1, \ldots x_m \in \HH$ in distinct $G$-orbits
(Definitions~\ref{idmap},~\ref{fieldmaps}) and for any field $L$ containing the field over
$E^{ab}(\Sigma)$ along with the coordinates of the $p(x_i)$,  the image of the
induced homomorphism $\rho: {\rm Gal}(\overline L/L) \rightarrow
\overline\Gamma^m$ has finite index in $ \overline \Gamma^m$.
\end{definition}

Recall from Lemma~\ref{3.3} that $$r_d(\vbar) \models \tp_{qf}(d/U).$$ where
$r_d(\vbar) =\bigcup_{\gbar \in G} \tp_{qf}(q(\gbar(d)) /U) =
\tp_{qf}(\langle q(g  d): g\in G\rangle/U)$.
The argument for Lemma~\ref{3.3} began with the observation that $r_d(\vbar)$
implied, in particular, that $d \not \in D_U$, so $d$ is an independent Hodge
generic.   We will deduce from Lemma~\ref{onto} that (under FIC1) only finitely many
tuples $\gbar$ from $r_d$ are really needed.



%
%
%
%


\begin{lemma} \label{onto} Assume FIC1.
Then, for each $z$, for some $\hat \gbar$,
%
 the map $$\rho_z: \aut(\overline{L}/L_{\hat \gbar}) \mapsto
  \overline{\Gamma}_{\hat \gbar}^{m}
  = \underleftarrow{lim}_{\hbar \supseteq {\hat \gbar}}({\Gamma_{\hat \gbar}/\Gamma_\hbar})^m$$ is surjective.
  \end{lemma}

\begin{proof} Let $I = \im (\rho_z)$ and let $k= [\overline{\Gamma}:I]$.  Suppose not. Choose $\hat{\gbar}$ with  $\gbar \subseteq
    {\hat{\gbar}}$ such that
$[\Gamma_{ \gbar} : \Gamma_{\hat\gbar}]=k$. Thus, for any $\hbar \supseteq
{\hat \gbar}$,  $\rho_z$ must be onto    $\Gamma_{\hat \gbar} /\Gamma_{
\hbar}$. For, if not, there is an $\eta \in \Gamma_{\hat \gbar} /\Gamma_{
\hbar}$ and that is not in $I$; it must be in a new coset of $I $ in
$\overline{\Gamma}$,
 contrary to the choice of $\hat {\gbar}$.
\end{proof}

\begin{corollary}\label{getprin} (Under FIC1) For $d \in D-U$,
$$\tp_{qf}(q(\hat \gbar(d))/U) \models r_d(\vbar) \models \tp_{qf}(  d/U).$$
\end{corollary}

\begin{proof} The second implication is Lemma~\ref{3.3}. For the first, choose
any $\hbar \supseteq \hat \gbar(d)$ and let $m= \lg(\hat \gbar)$,
 $r = \lg(\hbar)$. Let $\Fscr \subseteq Z^{r}_\hbar$ be the fiber over $\hat
\gbar(d')\in Z^{m}_{\hat \gbar}$ of the finite-to-one map $\psi_{\hbar \hat
\gbar }:Z^{r}_\hbar\rightarrow Z^{m }_{\hat \gbar}$. Similarly,
$\tp_{qf}(\hbar(d)/L_U )$ is determined by the $\aut(\CC/L_U)$-orbit $\Gscr
\subseteq \Fscr$ containing $\hbar(d)$.
Then, $\tp_{qf}(\hbar(\xbar)/L_U )$ is determined by the
$\aut(\CC/L_U)$-orbit $\Gscr \subseteq \Fscr$ containing $\hbar(\xbar)$.
But $\Gscr =\Fscr$, since
$\rho_z$ induces a homomorphism from $\aut(\CC/L_U)$ {\em onto} $\Gamma_{\hat
\gbar} /\Gamma_{\hbar}$ and $\Gamma_{\hat\gbar} /\Gamma_{\hbar}$ acts
transitively on the fiber. Since this holds for any such $\hbar$, we finish.
\end{proof}

We turn now to the infinitary axioms that are needed to obtain categoricity.

\begin{notation} [Infinitary Axioms]\label{infax}
\begin{enumerate}
\item $\Phi_\infty$ is the $L_{\omega_1,\omega}$ sentence asserting that
    for $(D, S, q)$ both the  dimension of  the field bi-interpretable with
    $S$ and of the strongly minimal  structure $\langle D, \{f_g: g \in
    \Gamma \}\rangle$ are infinite.

\item $SF$ (standard fibers) denotes the $L_{\omega_1,\omega}$-axiom:
$$(\forall x \forall y \in D)(q(x)= q(y) \rightarrow \bigvee_{g \in \Gamma}
x =f_g(y)).$$

\item $T^\infty (\pbar)$ denotes $\Th(\pbar) \cup \{\Phi_{\infty} \}$ and

\item $T^\infty_{SF}(\pbar)$ denotes $\Th(\pbar) \cup \{SF\}\cup \{\Phi_{\infty} \}$.

\end{enumerate}
\end{notation}

\begin{definition}\label{cldef} For $\langle D, S(F), q\rangle \models T^\infty_{SF}(\pbar)$ and $X \subset D \cup S(F)$,
$$\cl(X) = q^{-1}(\acl(q(X)) )$$ where $\acl$ is the field algebraic
closure in $F$.
\end{definition}

An essential consequence of the standard fibers axiom is that
Definition~\ref{cldef} defines an almost quasiminimal closure relation
satisfying the countable closure condition from Definition~\ref{qmdefstr}.
This closure dimension restricts on the separate sorts to
 the dimension of the constituent strongly minimal sets that is expressed in
 $\Phi_{\infty}$.
This accomplishes  the aim of an ($L_{\omega_1,\omega}$-complete so
$\aleph_0$-categorical) $L_{\omega_1,\omega}$ theory with arbitrarily large
models.

%
%
%



%


A class $\bK$ of models has $\aleph_0$-homogeneity over $\emptyset$  (Definition~\ref{qmdef})
(the precise statement is from \cite[p 4]{Etrev}) if the  models of $\bK$
are pairwise {\bf qf}-back and forth equivalent (Definition~\ref{bfdef}).


\begin{theorem}\label{fitohom} \cite[Theorem 4.11]{DawHarris} If the standard model $\pbar$ of a modular curve  satisfies FIC1, then the class of models of $T^{\infty}_{SF}(\pbar)$
is  $\aleph_0$-homogenous over   $\emptyset$.  In particular, by Karp~\cite{Karp, Barwisebf}, all models of $T^\infty_{SF}(\pbar)$ are back and forth equivalent and so satisfy the same sentences of $L_{\omega_1,\omega}$. 
\end{theorem}

\begin{proof}
 Our task is to replace the $\omega$-saturation hypothesis  from
Lemma~\ref{qe} by adding the infinitary axioms and the condition FIC1. As in the
 proof of theorem~\ref{qe} we need only worry about Hodge generic points.
 Suppose we have a partial function $f$ from $\qbar$ to $\qbar'$ with domain and
 range $U$ and $U'$ as in Lemma~\ref{qe} between models $\qbar$
and $\qbar'$ of $T^{\infty}_{SF}(\pbar)$. Proceed as in the proof of the second paragraph of
 Lemma~\ref{qe}. We vary  the argument for the `difficult case' from the 3rd paragraph.
  Choose $\hat \gbar$ by Lemma~\ref{onto}. Taking $\hat \gbar$ for the $\gbar$ in
  Lemma~\ref{onegbar}, for $x \in U_D$,
 there is   an $x'\in U_{D'}$ such that 
(*) $q(\hat \gbar (x'))
  \in S(F')^m$ realizes
$f(\tp_{qf} (q(\hat \gbar (x)) /L_U))$. We want to show that the same choice
$x'$ satisfies (*)
for every $\hbar\supseteq \hat \gbar $. 
This is immediate from Lemma~\ref{getprin}. The argument is completed by
induction as in the `final step' of the proof of Lemma~\ref{qe}.
\end{proof}



\begin{remark}[\rm  FIC2] \label{fic2}{\rm Like FIC1, FIC2 is a finite index condition on Galois
representations into inverse limits. Now, however there are independence
conditions over the ground field.  \cite[Condition 4.8]{DawHarris} provides
sufficient conditions so that a minor modification of the proof of
Theorem~\ref{fitohom}, shows FIC2 implies homogeneity over models; pairs of
models are back and forth equivalent over a countable submodel. {\em This is
the first place in the argument where types over countable algebraically
closed fields rather than the empty set (i.e.\ a fixed countable field) are
encountered.} Combining this result with Theorem~\ref{fitohom}, the
homogeneity conditions are now stronger than those defining quasiminimal
excellence in \cite{BHHKK}. Thus, we apply that paper and obtain:

}

\end{remark}

\begin{theorem}\label{mainres} For any modular curve interpreted as a standard model  $\pbar$
    (Definition~\ref{stmod}) for $T^\infty(\pbar)$,  $T^\infty(\pbar)$ is almost
    quasiminimal excellent and so categorical in every
    infinite power.
\end{theorem}

\begin{proof} We need only that FIC1 and FIC2 hold for all modular
 curves. This is proved in \cite[\S 5]{DawHarris}, where the proof
for FIC1 relies heavily on \cite[\S 6]{Serre} and FIC2 on \cite{Ribet}.
\end{proof}

 With
further effort they extend this result to Shimura curves.

\begin{remark} \label{compsentence}{\rm

 Keisler's theorem  \cite[Corollary 5.10]{Keislerlq}
	and work of Shelah \cite[\S 7]{Baldwincatmon}
	show that an $\aleph_1$-categorical sentence $\phi$ of
$L_{\omega_1,\omega}$ not only  has only countably many types in any
countable  fragment of $L_{\omega_1,\omega}$  containing $\phi$ (Keisler) but
has a completion\footnote{That is, a sentence $\phi^*$ that implies $\phi$ and decides every
$L_{\omega_1,\omega}$-sentence.} (Shelah).  Equivalently, the completion must
specify the isomorphism type of the countable model. The only such completion
consistent with having an uncountable model is adding $\Phi_\infty$.

We have used FIC1 to prove categoricity in all powers. In fact,
$\aleph_1$-categoricity implies FIC1. For this, \cite{DawHarris, Etrev} argue
that the weaker hypothesis of having just countably  many types over the
empty set in the theory  $T^{\infty}_{SF}$ implies FIC1. If FIC1 holds, for
some $z$, by Lemma~\ref{onto}, for every  $\gbar$, there is $\hbar \supseteq
\gbar$ with a $\Gamma_\gbar/\Gamma_\hbar$-orbit contained in
$\psi_{\hbar\gbar}^{-1}(z)$ that projects to that $\Gamma_\gbar$ orbit. So
under the assumption that FIC1 fails, there is a $\gbar$, such that for every
$\hbar \supseteq \gbar$ there are distinct $\Gamma_\gbar/\Gamma_\hbar$-orbits
$O_1,O_2$ contained in $\psi_{\hbar\gbar}^{-1}(z)$ that project to the same
$\Gamma_\gbar$-orbit.

By Lemma~\ref{3.3}, if
two points are Galois equivalent they realize  the same quantifier free $\tau$-type;
so $O_1,O_2$ realize distinct Galois orbits (and so any two orbits that project
 to them must realize distinct $\tau$-types). But
 since $\overline \Gamma$ acts transitively on each $Z_\gbar$, there is a complete
  tree of splittings of $\aut(\CC/L)$ orbits that all project to $z$.
   This contradicts Keisler's theorem. So $\aleph_1$-categoricity
   of $T^{\infty}_{SF}$ implies FIC1.

}\end{remark}
%
%


\begin{remark} \label{ficstatus} {\rm
 \cite[\S 5]{DawHarris}, using both Serre's open mapping theorem \cite[\S 6]{Serre} for the finite index condition and work by \cite{Ribet} on Shimura curves show
FIC1 and FIC2 hold for all modular and Shimura curves.  So our remaining
sections concern  higher dimensional varieties. FIC1 is known for some higher
dimensional Shimura varieties and conjecturally for others, while FIC2 is
true for all \cite{Etrev}.


\cite{DawHarris} use both to prove categoricity. Since the Galois group is
not accessible in our formal language, FIC1 cannot be directly expressed in
the two-sorted theory. So the goal of a `fully formal invariant' cannot be
achieved unless
 explicit reliance on the finite index conditions as an hypothesis is avoided.}
\end{remark}

 \section{First order Excellence}\label{section:notop}

Here is the opening paragraph of \cite{BHP}.

\begin{quotation}
\emph{
    Let $\GG =\GG^n$ be a complex algebraic torus, or let $\GG$ be a complex abelian variety.
Considering $\GG(\CC)$ as a complex Lie group,
with $\bL\GG = T_0(\GG(\CC))$ its (abelian)
Lie algebra, the exponential map provides a
surjective analytic homomorphism
$$\exp: \bL\GG \twoheadrightarrow \GG(\CC).$$
}
 \end{quotation}

In the spirit of Zilber, their paper aims at finding `algebraic descriptions' of the cover $\exp$ which
characterize the standard structure (at least up to categoricity in power).  They solve a more
general problem by providing a first order theory $\hat T$ for the situation and showing each  model
$\tilde M$  ($\hat M$ here) of $\hat T$ is determined by  relations among two designated
substructures and a certain transcendence degree. In this generality, the result is proved for any
abelian group of finite Morley rank (henceforth fmr groups).  Then, under slightly stronger
hypotheses, the result becomes a true categoricity result for, in particular, an abelian variety
defined over a number field.

We address in this section  four new ingredients: formalized non-standard
covers, `first order excellence', Kummer theory, and a distinction between
classification and categoricity. First order excellence  appears to be both
necessary and applicable for higher order Shimura varieties.

As noted in \cite{BHP}, the quasiminimal approach studied earlier in this paper suffices to prove
the $L_{\omega_1,\omega}$-categoricity in power for Abelian varieties. The goal of this section is
to identify the distinctive elements of the \cite{BHP} proof that later reappear in \cite{Etrev}.

\subsection{The two-sorted structure and fmr groups}\label{2sort}

A first order theory $T$ is stable in $\kappa$ if any $M \models T$, with
$|M|=\kappa$, $|S(M)| = \kappa$. ($S(M)$ denotes the set of $1$-types over
$M$.) Morley showed that $\omega$-stability (more properly,
$\aleph_0$-stability) of a theory $T$ is equivalent to stability in all
powers (and also to the Morley rank having an ordinal value for each type).
We need here a slightly weaker condition called {\em superstability}: $T$ is
stable in $\kappa$ if $\kappa \geq 2^{\aleph_0}$.

The theory of $(\mathbb{Z},+)$ is one of the prototypical strictly
superstable theories\footnote{The other one is the theory of countably many
equivalence relations $E_n$ such that for each $n$, each $E_n$-class is split
into infinitely many $E_{n+1}$-classes (and $E_{n+1}\subseteq E_n$).} (that
is, superstable, but not $\aleph_0$-stable). One can fix arbitrarily the
congruence class of an element $x$ for each $n$. This gives   $2^{\aleph_0}$
distinct types realized by  non-standard integers.

There is an extensive theory of fmr groups (see~\cite{BorovikNesinbook,ABC}).
We need here only the basics. In particular, Macintyre's result \cite{Macab}
that an $\omega$-stable group is divisible by finite. We now introduce the
two-sorted theory; with that notation we are able at the end of this section
to outline the main steps of the proof.

Unlike \cite{DawHarris} where  $\underleftarrow{\lim}\ Z_\gbar$ is in the
background of the proof of (our)    Theorem~\ref{fitohom} but not the
statement, \cite{BHP} build the structure of non-standard covers into the
vocabulary of the two sorted structure by the $\rho_n$ below.

\cite[\S 2.2]{BHP} use the  inverse limit of Definition~\ref{bhphat} for {\em
divisible} abelian groups; although it is not profinite, they refer to it as
a profinite universal cover denoted $\hat G$ of $G$ and $G$ is renamed as
$M$. Although the hat has only one meaning in \cite{BHP}, it becomes
overloaded here so we denote the inverse limit defined below as $\tilde M$.
While in \cite{BHP} a typical 2-sorted (3-sorted in \S \ref{smoothvar})
structure $\hat \tau$ is represented as either $(\tilde M, M)$ or $\tilde M$,
we write $\hat M = (\tilde M, M)$ and $\tilde M$ for the  or (profinite
cover) inverselimit  from \cite[1.2, 2.1]{BHP}  as that is the actual usage
in most of the cited paper.


\begin{definition} [$\tilde M$]\label{bhphat}
Given a commutative, divisible, abelian group $(M,+)$, consider the inverse
limit $\tilde M=  \underleftarrow{\lim}\ M_n$ of isomorphic copies $M_m$ of
$M$ with the index set partially ordered by $m\leq n$ if and only $m|n$ and
with maps $\eta_{nm}$ (multiplication by $\frac{m}{n}$) taking $M_n \mapsto
M_m$. Concretely, $(\tilde M,+)$ is the subgroup of the direct product of
$\omega$ copies of $M$,
containing those sequences ($\langle g_k: 1 \leq k< \omega\rangle$) such that
if  $k=nm$, $g_m= n \times g_k$ and $g_n = m\times g_k$. \end{definition}

\begin{notation} [The vocabulary $\hat \tau$]\label{bhpvoc}
Let $\GG$ be the given abelian group and $T:= \th(\GG)$ in a large enough
countable language that $T$ has quantifier elimination. Further, let $\hat T$
be the theory of $(\hat\GG, \GG)$ in the two-sorted language $\hat \tau$
consisting of the maps $\rho_n: \hat\GG \rightarrow \GG$  for each $n$, the
theory $T$ and, for each $\acl^{eq}(\emptyset)$-definable subgroup $H$ of
$G$, a predicate $H$ for $H$ and a predicate $\hat H$ for $\{x \in \hat G:
\rho_n(x) \in H, n\in\NN\}$.

Although the kernel of $\rho= \rho_1$ is definable in the vocabulary given, a
further predicate $\ker^0$ is included denoting the divisible part of the
kernel (otherwise, it is only type-definable).
\end{notation}

The axioms \cite[2.5]{BHP} of $\hat T$ are chosen  so that

\begin{theorem}\cite[2.7, 2.8, 2.21]{BHP}\label{hatTthm}
For an fmr group $\GG$, $(\underleftarrow \GG,\GG,\rho_0) \models \hat T$ and
therefore $\hat T$  admits quantifier elimination and is superstable of
finite $U$-rank.
\end{theorem}

 Although the $T$ in Notation~\ref{bhpvoc} is $\omega$-stable, $\hat T$
is only superstable; also, many elements of $\ker(\rho)$ are not divisible in $\ker(\rho)$.




\begin{remark}[Quasiminimality, unidimensionality, notop] \label{kf} {\rm Abelian varieties
 as opposed to fmr groups, can be
	handled either by the quasiminimality methods of Section~\ref{modshicurves} or by the
	methods described in this section.
A crucial distinction from Section~\ref{modshicurves} is that the former considered only the  theory
of  unary functions from a group acting on the domain, while here we
have the full group structure.

To explain the fmr proof we need some further model theoretic background. In
general two   types $p,q$ over $M$ are {\em orthogonal} when in different
models $N$ extending $M$ the number of realizations of $p$ and $q$ can be
varied arbitrarily. {\em Non-orthogonality} for strongly minimal sets has a
particularly clear meaning. The strongly minimal sets $D_1$ and $D_2$ are
non-orthogonal if there is a definable finite to finite binary relation on
$D_1 \times D_2$. A theory is unidimensional if all types are non-orthogonal.

The three features that  underlie the \cite{BHP} proof   are.

\begin{enumerate}
\item A fmr abelian group has {\em finite width} \cite[XV.1]{Baldwinbook}
    (aka almost $\aleph_1$-categorical \cite{Lascargrp}): Any model is the
    algebraic closure of the union of the bases of  a collection of
    strongly minimal $D_i$ for $i< n<\omega$. The $D_i$ are defined over
   the prime model (the unique up to isomorphism model elementarily
   embedded in every model of the theory).

    \item  In models of $\hat T$ with  $M_0$ the prime model of $T$ and
        where $\GG$ is defined over a number field $k_0$, Kummer theory
	allows the control of $\rho^{-1}(M_0)$ by  the kernel $\rho^{-1}(0)$.

        \item In studying  Abelian varieties the $n$ in 1) can be taken as
            $1$ because the variety is interalgebraic with an algebraically
            closed field and
	so {\em almost strongly minimal}
        ($M =\acl (D)$ for strongly minimal $D$).
        \end{enumerate}

Since Kummer theory doesn't apply to arbitrary Shimura varieties, both 2) and
3) fail for
  more general higher dimensional Shimura varieties  (see Section~\ref{Shivar}).}
  \end{remark}


\subsection{First order Excellence and fmr groups}\label{foex}

Shelah's main gap program defines a sequence of properties $X$ of countable
first order theories forming a sequence of dichotomies \cite[\S
5.5]{Baldwinphilbook} such that: if $T$ satisfies $X$, $T$ has the maximal
number of models in every uncountable cardinal. If $T$ fails $X$, the models
of $T$ satisfy conditions useful for classification. (e.g. stability implies
the existence of the `non-forking' independence relation). The positive side
of the final  dichotomy in the sequence is superstable {\em without the
omitting types order property} (denoted \textbf{notop}).  Under this
hypothesis, Shelah (\cite{Shelahbook2nd} and earlier papers) showed that an
appropriate class of models of $T$ had a notion of independence among
structures with $n$-amalgamation for all $n$ that yields the classification
of models. Hart \cite{Hart} reduced the amalgamation requirement to
$2$-amalgamation and this reduction was extended to the quasiminimal
excellent case in \cite{BHHKK}.
In Section~\ref{Shivar}, we note this `notop' approach is used to study higher dimensional Shimura
varieties.

In Section 3 of \cite{BHP} the techniques of \cite{Hart} are adapted to the
specific framework here to establish a  decomposition of models  of $\hat T$
analogous to that in Remark~\ref{kf}  for models of $T$. This yields

\begin{theorem} \cite[Theorem 3.31]{BHP} \label{3.31}
 Each model $\hat \Mscr$ of $\hat T$ is determined up to
isomorphism by the transcendence degree of the algebraically closed field $K$ such
that $M \cong \GG(K)$, the isomorphism type of the inverse image, $\hat{M_0}$, of the prime model
$M_0$ of $T$,
and the isomorphism type of $M$ over $M_0$.
\end{theorem}

\subsection{Abelian Varieties}\label{abelianvarieties}

From the model theoretic standpoint, an {\em Abelian variety} is a complete
algebraic variety whose points form a group such that the group operations
are definable in the ambient field. For Abelian varieties, Kummer theory
eliminates (as in \cite{Gavk,BGH}) the reliance in Theorem~\ref{3.31} on
knowing the isomorphism type of $\check M_0$ over the kernel. The situation
described in the opening paragraph of \S~\ref{section:notop} is a special
case.  Namely, let $\GG$ be (the formula defining) an abelian variety
$\GG(K)$ over a field $K$ as in the introduction to
Section~\ref{section:notop}. Assume $\GG(\CC)$ and its ring of endomorphisms
are definable over a number field $k_0$. With this notation:

\begin{theorem} \cite[Theorem 4.6]{BHP} \label{bhp4.6}
 a model $\hat M =\langle \tilde M, M,q\rangle $ of $\hat T$ is determined up to
isomorphism by the transcendence degree of the algebraically closed field $K$
such that $M \cong \GG(K)$, and the $\hat \tau$ isomorphism type of $\ker
\rho$.
\end{theorem}


\begin{remark}[Complete formal invariant]\label{cfi}
{\rm Theorem~\ref{bhp4.6} gives categoricity in all uncountable cardinalities
 by adding the $L_{\omega_1,\omega}$ sentence characterizing  the standard kernel.
  But Theorem~\ref{bhp4.6} is more general than categoricity; it shows that models with
  non-standard (possibly uncountable) kernel are characterized by the
  $\hat\tau$-diagram of the kernel.  Of course, this statement cannot be
  formalized in languages with bounded
length of conjunctions since the kernels can be arbitrarily large. But
Zilber's goal (just after Notation~\ref{gs}) only aimed at complete formal
characterization for prototypical mathematical structures. }
\end{remark}

\section{Higher Dimensional Shimura Varieties}\label{Shivar}

A Shimura variety is a higher-dimensional generalization  of a modular curve
that arises as a quotient variety of a Hermitian symmetric space $X^+$ by a
congruence subgroup of a reductive algebraic group defined over $\QQ$. We
consider Shimura varieties that are moduli spaces for generalized algebraic
varieties.
Rather than discussing   further
technical details on the definition of a Shimura datum $(G,X)$, we survey the
differences that arise in generalizing the results in Remark~\ref{ficstatus}
about Shimura curves to higher dimensional Shimura varieties: $S(\CC) =
\Gamma \setminus X^+$.

Central difficulties arise directly from the higher dimension in two ways.  First, in the curve case
the $2$-sorted structure is (almost)-quasiminimal because the variety in field sort is a curve and
so strongly minimal and the
geometric closure on the cover sort is given by $a \in \cl(X)$ if $a \in q^{-1} (\acl((q(X))$.
 Quasiminimality can fail in the higher dimensions.  Second, rather than special points which are
 fixed points of some $g$, one must treat {\em special subvarieties} \cite[\S 3.4]{Etrev} and finite
 unions thereof, {\em special domains}. The fact that these are not merely points leads to several
 difficulties.

\begin{enumerate}
\item The structure of the covering sort is no  longer strongly minimal. Even after naming the
    elements of the group the special subvarieties give a complicated structure on the covering
    sort.
    \item In the curve case the intersection of special domains was a point; that may fail in higher
	dimensions.
        \item The theories of two inverse limit structures $\hat \pbar$ and
            $\tilde \pbar$ are
	    considered as the covering space. The  first structure is the analog of
	    $\underleftarrow{\lim}\  Z_{\gbar}$ (Definition~\ref{zgbar}). The second consists only of the standard points of this limit. The
	    canonical universal cover $\pbar$ satisfies the first order $\th(\tilde \pbar)$ but not
            in general $\th(\hat\pbar)$ \cite[Example 5.7, Corollary 5.14]{Etrev}.
            \item An $L_{\omega_1,\omega}$ categorical axiomatization is not claimed. Each model can
		be precisely characterized but the characterization is not in $L_{\omega_1,\omega}$.
		See Remark~\ref{cfi}.

                \item Finally, even this characterization depends on whether the variety under
		    consideration satisfies {\em finite index conditions} as in the modular  case.
		      Although FIC1 and FIC2 are true in the modular curve case,
                    here  the truth of $FIC1$ for $\pbar$ is actually equivalent
		    to the characterizability  of models of $T^{\inf}_{SF}(\pbar)$
since \cite{Etrev} shows FIC2 is true.
     \end{enumerate}












\section{Model Theory and Analysis}

 One can signal three different
model theoretic approaches to analysis:

\begin{enumerate}
\item {\bf Axiomatic analysis}  studies behavior of fields of functions
    with  operators but {\em without} explicit attention in the formalism
    of continuity but rather to the algebraic properties of the functions.
    The function symbols of the vocabulary act on the functions being
    studied; the functions are elements of the domain of the model.

    Example: $DCF_0$ as discussed below.
\item {\bf Definable analysis} has a lower level of
abstraction; the domain of the functions remains the universe of the
model. The functions being studied are the compositions of the
functions named in the vocabulary; one cannot quantify over them.

Example: $o$-minimality.
\item {\bf Implicit analysis}  Attempts to provide `algebraic
    characterizations of important mathematical structure by
    axiomatizations in infinitary logic that are categorical in power.
Example: the material in this paper.
\end{enumerate}


The first two are discussed in \cite[\S 6.3]{Baldwinphilbook}.
The work expounded in this paper has many commonalities  with a prime example of
axiomatic analysis: the study of
transcendence results for solutions of differential equations
by the study of the $\omega$-stable theory $DCF_0$ of differentially
closed fields of characteristic zero.  The notion of `not integrable  by elementary functions
 (Painlev\'e said `irreducible') is formalized by `the solution set is strongly minimal'
 \cite{Nagloothesis}. The study of Schwartzian
equations provides a general framework in which the $j$-function and modular
curves are explored. The work includes, variations on the
Ax-Lindemann-Weierstrass theorem, proofs that Generic differential equations
are strongly minimal \cite{DevilF} and
 Differential Chow Varieties are Kolchin-constructible
\cite{FLS}, and  analysis of strongly minimal solution sets  defined by
differential equations in terms of the Zilber trichotomy and
$\aleph_0$-categoricity.

But while the mathematical topics are the same, the aims are different: The
covers project tries to assign a categorical description of each cover. The
$DCF_0$ approach tries to understand transcendence results for solutions of
the differential equations.

The crucial methodological difference is the two-sorted nature of the cover
program.  The axiomatic analysis framework is preserved in that there is no
explicit treatment of convergence or continuity. But connecting the domain
and target by quotients under an explicit group action as well as the use of
infinitary logic provides tools not available in the earlier examples of
axiomatic analysis.

%
%


\section{Families of covers of algebraic curves}\label{smoothvar}

%
%
%
%
%

 In recent work Zilber and Daw ~\cite{DawZil0,DawZil1} deal with \emph{families} of
covers of curves. They build on earlier constructions we have discussed in
this paper. Rather than a cover of a \emph{single} variety, albeit one that
parameterized a family of varieties,  an entire family of such covers is
studied and the covering space becomes an {\em analytic Zariski structure}
\cite{ZilberZariski}.  In \cite{Zilbomin} the analysis of families is
generalized by being placed in a geometric algebraic setting.

The most salient difference between these works and those discussed earlier
in this paper is that, rather than a cover of a \emph{single} variety, an
\emph{entire family of covers} is now the main subject.  Our earlier
Definition~\ref{disgrpterm} is now replaced by a basic vocabulary consisting
of \emph{three} sorts, together with maps $\Gamma_N \setminus \HH \mapsto
\CC$ covering a family of curves $S_N(\CC)$.


\subsection{Pseudo-analytic covers of modular curves}\label{families}

 Major differences of
paper~\cite{DawZil0} from the earlier discussion of modular curves include:

\begin{enumerate}
    \item The basic vocabulary is now $3$-sorted.
More specifically, \cite{DawZil0} considers structures $(D, G, j_N, \CC)$
where the $j_N: \HH \twoheadrightarrow S_N(\CC)$.
	The discrete group is
        now given as a third sort incorporating \emph{a group operation}
	(so its pregeometry is locally modular, rather than trivial). This sort
contains group with distinguished subsets\footnote{$E$ is the elliptic
Möbius transformations and the $\dbar_{\qbar},\dbar_{\qbar}$ are specific
{\em diagonal} matrices.} $(\gl^+_2(\QQ), \times, \sl_2(\ZZ), E(\QQ),
\{\dbar_q, \dbar'_q: q \in \QQ)\}$, where $E$ is the collection of elliptic
elements of the group; those that have unique fixed points.
	This structure is  specified up to isomorphism  by a sentence of $L_{\omega_1,\omega}$.
	But not all group elements are still named in the formal language.

    \item  \label{sorted}   The uniformizing functions $j_N$ each map into
        ${\mathbf P}^3(\mathbb{C})$ rather than
	into the arbitrarily high dimensional spaces of the maps $[\phi_{\gbar}]$
	in~\cite{DawHarris,Etrev}. Furthermore, these are now defined over
 $\mathbb{Q}$ rather than
	over $E^{ab}(\Sigma)$.

    \item As well as an almost quasiminimal axiomatization of the
        $3$-sorted structure, the domain is considered as a Zariski
        Analytic set
	 with a quasiminimal geometry. Both of these structures are shown to be uncountably categorical.
    \item The special points \emph{are not named}. However as in
        Definition~\ref{mcax} they are uniquely associated with elliptic
        elements of the group.
\end{enumerate}

In many ways, this last distinction is the most important for the general
program, as naming of the special points trivializes some of the arithmetic.
In~\cite{DawZil0}, the {structure} of the \emph{family} is proved to be
categorical in all uncountable cardinalities.

\subsection{Locally o-minimal covers of algebraic varieties }\label{omin}

The paper~\cite{Zilbomin} takes a {\em more general} approach. It abstracts
away from  naming all elements of the discrete groups as earlier in this
paper. The relations among the universal and finite covers are given more
abstractly as properties of maps from a domain (whose smoothness is defined
topologically and geometrically but not algebraically) onto families of
algebraic varieties. This smoothness as well as the eventual quasiminimality
for curves\footnote{The set-up is for arbitrary algebraic varieties, but the
categoricity result is only for curves and we restrict to that case.} is
controlled by \emph{external} o-minimal structures.

%

\begin{remark}\label{whatsnew}
{\rm
\begin{enumerate}
 \item The formalization is new. For a fixed  model  $\mathsf R$ of the
     theory $T$ of a fixed o-minimal expansion of the reals (e.g the
     restricted analytic functions) a structure $\UU(\mathsf R)$ is
     defined. The resulting structure $\UU(\mathsf R)$ is
an abstract Zariski structure\footnote{Actually, $\UU(\mathsf R) = U(K)$
where $\mathsf K$ is taken as an algebraically closed field ${\mathsf
R}+i{\mathsf R}$ and $U(R)$ is constructed analogously to $U(\CC)$.}.
    \item Generalizing the last paragraph of Section~\ref{families}, in the
        standard model the domain is a complex manifold $\UU(\CC)$ with
        holomorphic maps $f_i$ onto algebraic varieties $X_i(\CC)$ with
	natural projections $\pr_{i,j}$ among the $X_i$. These analytic
properties are definable using  theory of $\mathsf K$-analytic sets in
o-minimal expansions of the reals developed in \cite{PScom08, PSicm10}.  We
fix
      $k\subseteq \mathbb{C}$ a subfield
	over which the varieties $\mathbb{X}_i$ are all defined.

\item The ostensibly  two-sorted structure of 1) becomes one-sorted because
    the field can be interpreted in the abstract Zariski structure. And the
    third sort of Section~\ref{families} has disappeared because the group
    is no longer referenced directly.
%


\item The $o$-minimal geometry of algebraic closure in $\UU(\mathsf R)$
    imposes the desired quasiminimal geometry on $\UU(\mathsf R)$. The
    dimension function is denoted ${\rm cdim}$ for `combinatorial
    dimension'.  Note that the ordering is not externally imposed on $\UU$: rather, it is implied
    by the predicates described in (1) above and the dimension just mentioned.
%
    \item As before, there is an $L_{\omega_1,\omega}$ sentence that
        axiomatizes the quasiminimal (excellent) geometry and whose models
        form an AEC that is categorical in all cardinalities.

%
\end{enumerate}
}
\end{remark}

Zilber provides a proof  of the following theorem \cite{Zilbomin}:

\begin{theorem}[Categoricity of families of smooth complex algebraic varieties~\cite{Zilbomin}]
    Let $\mathbb{U}$ be a cover of a family of smooth complex algebraic variety,
    formalized as in Remark~\ref{whatsnew}, and
    let $\mathfrak{U}(\mathsf R)$ be its associated $L_{\omega_1\omega}$-definable class.
    If $\dim_{\mathbb{C}}(\mathbb{U})=1$, (i.e. if the varieties are curves) and ${\rm cdim}(\mathsf
    R/k)$ is infinite, then
    $\mathfrak{U}(\mathsf R)$ is
    categorical in all uncountable cardinals.
\end{theorem}

Zilber remarks   that in the case of higher dimensional varieties,
categoricity in $\aleph_1$ can still be proved.

%
%
%
%





\begin{example}{\rm 

Here are some examples from \cite{Zilbomin}.
Fix the o-minimal expansion $\mathbb{R}_{\rm An}=\mathbb{R}_{\exp,an}$ of the
reals with the exponential function and the restricted (to bounded intervals)
analytic functions.

\begin{itemize}
    \item Let $I=\mathbb{N}$, $\mathbb{U}=\mathbb{C}$, $f_k(z)=\exp(\frac{z}{k})$, $D_n=\{z
	\in \mathbb{C}:-2\pi n<{\rm Im}(z)<2\pi n\} $. These are easily seen to provide a cover
	system.
    \item 
     The  $j$-function with variants $j_N$ as uniformizers for the modular
        curves  $\Gamma_N \setminus \HH$
are examples; this study
	allows one to formalize their analytic properties in terms of o-minimality. Finally, other
	examples include the Siegel half-space and polarized algebraic varieties (these last
	examples are claimed but not developed by Zilber).
\end{itemize}}
\end{example}



\begin{thebibliography}{BMPTW20}

\bibitem[ABC08]{ABC} T.~Altinel, A.~Borovik, and G.~Cherlin.
\newblock {\em Simple Groups of Finite Morley Rank of Even Type}.
\newblock Mathematical Surveys and Monographs. American Mathematical Society,
  2008.

\bibitem[Bal88]{Baldwinbook} John~T. Baldwin.
\newblock {\em Fundamentals of Stability Theory}.
\newblock Springer-Verlag, 1988.

\bibitem[Bal09]{Baldwincatmon} John~T. Baldwin.
\newblock {\em Categoricity}.
\newblock Number~51 in University Lecture Notes. American Mathematical Society,
  Providence, USA, 2009.

\bibitem[Bal18]{Baldwinphilbook} John~T. Baldwin.
\newblock {\em Model Theory and the Philosophy of Mathematical Practice:
  Formalization without Foundationalism}.
\newblock Cambridge University Press, 2018.

\bibitem[Bar73]{Barwisebf} J.~Barwise.
\newblock Back and forth through infinitary logic.
\newblock In M.~Morley, editor, {\em Studies in Model Theory}, pages 1--34.
  Mathematical Association of America, 1973.

\bibitem[BGH14]{BGH} M.~Bays, M.~Gavrilovich, and M.~Hils.
\newblock Some definability results in abstract {K}ummer theory.
\newblock {\em International Mathematics Research Notices}, 43:3975--4000,
  2014.
\newblock \url{https://doi.org/10.1093/imrn/rnt057}.

\bibitem[BHH{\etalchar{+}}14]{BHHKK} M.~Bays, B.~Hart, T.~Hyttinen,
    M.~Kesala, and J.~Kirby.
\newblock Quasiminimal structures and excellence.
\newblock {\em Bulletin of the London Mathematical Society}, 46:155--163, 2014.

\bibitem[BHP20]{BHP} M.~Bays, B.~Hart, and A.~Pillay.
\newblock Universal covers of commutative finite {M}orley rank groups.
\newblock {\em Journal of the Institute of Mathematics of Jussieu},
  19:767--799, 2020.
\newblock \url{https://www3.nd.edu/~apillay/papers/universalcovers-BHP.pdf}.


\bibitem[BZ11]{BaysZil} M~Bays and B~Zilber.
\newblock Covers of multiplicative groups of algebraically closed fields of
  arbitrary characteristic.
\newblock {\em Bull. Lond. Math. Soc.}, 43:689--702, 2011.

\bibitem[BZ08]{modsp} D.~Ben-Zvi.
\newblock Moduli spaces.
\newblock In T.~Gowers, editor, {\em The Princeton Companion to Mathematics},
  pages 408--419. Princeton University Press, 2008.

\bibitem[BHV18]{BerensteinHyttinenVillaveces} A.~Berenstein, T.~Hyttinen, and
    A.~Villaveces.
\newblock Hilbert spaces with generic predicates.
\newblock {\em Rev. Col. Mat.}, 52:107--130, 2018.

\bibitem[BN94]{BorovikNesinbook} A.~Borovik and A.~Nesin.
\newblock {\em Groups of {F}inite {M}orley {R}ank}.
\newblock Oxford University Press, 1994.

\bibitem[BMPTW20]{BreuillardPizarroTentWagner} E.~Breuillard,
    A.~Martin-Pizarro, K.~Tent, and F.O. Wagner.
\newblock Model theory: Groups, geometries and combinatorics.
\newblock {\em Oberwolfach Rep.}, 17(1):91–142, 2020.






\bibitem[CMVZ21]{CrViZi} John~Alex Cruz-Morales, Andres Villaveces, and Boris
    Zilber.
\newblock Around logical perfections.
\newblock {\em Theoria}, 87:971--984, 2021.



\bibitem[DH17]{DawHarris} Christopher Daw and Adam Harris.
\newblock Categoricity of modular and {S}himura curves.
\newblock {\em Journal of the Institute of Mathematics of Jussieu},
  66:1075--1101, 2017.
\newblock math arxiv \url{https://arxiv.org/pdf/1304.4797.pdf}.

\bibitem[DZ22a]{DawZil1} C.~Daw and B.I. Zilber.
\newblock Canonical models of modular curves and galois action on cm-points.
\newblock math arxiv: \url{https://arxiv.org/pdf/2106.06387.pdf}, 2022.

\bibitem[DZ22b]{DawZil0} C.~Daw and B.I. Zilber.
\newblock Modular curves and their pseudo-analytic cover.
\newblock math arxiv: \url{https://arxiv.org/pdf/2107.11110.pdf} Nov., 2022.

\bibitem[DF23]{DevilF} M.~DeVilbiss and J.~Freitag.
\newblock Generic differential equations are strongly minimal.
\newblock to appear Compositio Math; Math Arxiv:
  \url{https://arxiv.org/abs/2106.02627}, 2023.

\bibitem[Ete22]{Etrev} Sebastian Eterovi\'{c}.
\newblock Categoricity of {S}himura varieties.
\newblock 2022 version, 2022.

\bibitem[FLS17]{FLS} J.~Freitag, Wei Li, and T.~Scanlon.
\newblock Differential chow varieties exist.
\newblock {\em J. Lond. Math. Soc.}, (2) 95:128--156, 2017.
\newblock appendix by William Johnson.

\bibitem[Gav08]{Gavk} Misha Gavrilovich.
\newblock A remark on transitivity of {G}alois action on the set of uniquely
  divisible abelian extensions of ${E}(\overline{{Q}})$ by $z^2$.
\newblock {\em Journal of {K}-theory}, 38:135--152, 2008.

\bibitem[GL02]{Bilgi} R.~Grossberg and Olivier Lessmann.
\newblock Classification theory for abstract elementary classess.
\newblock In {\em Logic and Algebra}, volume 302 of {\em Contemporary
  Mathematics}, pages 165--204. AMS, 2002.

\bibitem[Had54]{Hadamard} Jacques Hadamard.
\newblock {\em The psychology of invention in the mathematical field}.
\newblock Dover, 1954.
\newblock First edition Princeton 1945; Dover:
  \url{http://worrydream.com/refs/Hadamard\%20-\%20The\%20psychology\%20of\%20invention\%20in\%20the\%20mathematical\%20field.pdf}
  French version of Poincaire story
  \url{https://www.persee.fr/doc/ahess_0395-2649_1963_num_18_2_420994_t1_0399_0000_1}.

\bibitem[Har87]{Hart} Bradd Hart.
\newblock An exposition of {O}{T}{O}{P}.
\newblock In J.~Baldwin, editor, {\em Classification {T}heory: {C}hicago,
  1985}. Springer-Verlag, 1987.

\bibitem[Har14]{Harristhe} Adam Harris.
\newblock {\em Categoricity and covering spaces}.
\newblock PhD thesis, Oxford, 2014.
\newblock \url{https://arxiv.org/pdf/1412.3484.pdf}.

\bibitem[Kar64]{Karp} C.~Karp.
\newblock {\em Languages with Expressions of Infinite Length}.
\newblock North Holland, 1964.

\bibitem[Kat92]{Katok} S~Katok.
\newblock {\em Fuchsian Groups}.
\newblock University of Chicago Press, Chicago, 1992.

\bibitem[Kei70]{Keislerlq} H.J. Keisler.
\newblock Logic with quantifier "there exists uncountably many".
\newblock {\em Annals of Math. Logic}, 1:1--93, 1970.

\bibitem[Kir10]{Kirbyqm} Jonathan Kirby.
\newblock On quasiminimal excellent classes.
\newblock {\em Journal of Symbolic Logic}, 75:551--564, 2010.

\bibitem[Las85]{Lascargrp} D.~Lascar.
\newblock Les groupes $\omega$-stable de rang fini.
\newblock {\em Transactions of the American Mathematical Society},
  292:451--462, 1985.

\bibitem[Mac70]{Macab} Angus~J. Macintyre.
\newblock On $\omega_1$-categorical theories of abelian groups.
\newblock {\em Fundamenata Mathematicae}, 70:253--270, 1970.

\bibitem[Mar02]{Markerbook} D.~Marker.
\newblock {\em Model Theory: An {I}ntroduction}.
\newblock Springer-Verlag, 2002.

\bibitem[Mil12]{milnenams} James~S. Milne.
\newblock What is a {S}himura variety.
\newblock {\em Notices of the American Mathematical Society}, 59:2560--1561,
  2012.

\bibitem[Miy89]{Miyake} T.~Miyake.
\newblock {\em Modular Forms}.
\newblock Springer, 1989.

\bibitem[Mor65]{Morley65} M.~Morley.
\newblock Categoricity in power.
\newblock {\em Transactions of the American Mathematical Society},
  114:514--538, 1965.

\bibitem[Nag14]{Nagloothesis} Joel Nagloo.
\newblock {\em Model theory and differential equations}.
\newblock PhD thesis, Leeds, {2014}.



\bibitem[PS98]{PetStartri} Ya'acov Peterzil and Sergei Starchenko.
\newblock A trichotomy theorem for o-minimal theories.
\newblock {\em Proceedings of the London Mathematical Society}, 77:481--523,
  1998.

\bibitem[PS08]{PScom08} Ya'acov Peterzil and Sergei Starchenko.
\newblock Complex analytic geometry in a non-standard setting.
\newblock In Chatzidakis, MacPherson, Pillay, and Wilkie, editors, {\em Model
  Theory with Applications to Algebra and Analysis}, Lecture Notes Series 349
  I, pages 117--166. London Math Soc., 2008.

\bibitem[PS10]{PSicm10} Ya'acov Peterzil and Sergei Starchenko.
\newblock Tame complex analysis and o-minimality.
\newblock In {\em Proceedings of the International Congress of Mathematicians,
  Vol. II (New Delhi, 2010)}, pages 58--81, 2010.

\bibitem[Poi85]{Poizatbook} Bruno Poizat.
\newblock {\em Cours de Th\'eories des Mod\`eles}.
\newblock Nur Al-mantiq Wal-ma'rifah, 82, Rue Racine 69100 Villeurbanne France,
  1985.

\bibitem[Poi87]{Poizatbook2} Bruno Poizat.
\newblock {\em Groupes Stables}.
\newblock Nur Al-mantiq Wal-ma'rifah, 82, Rue Racine 69100 Villeurbanne France,
  1987.

\bibitem[Rib75]{Ribet} K.~Ribet.
\newblock On $\ell$-adic representations attached to modular forms.
\newblock {\em Inventiones Mathmematica}, 28:245--275, 1975.

\bibitem[Ser72]{Serre} J.~P. Serre.
\newblock Propi\`{e}t\`{e}s galoisienne des pointes d'ordre fini des courbes
  elliptiques.
\newblock {\em Invent. Mat.}, 15:259--331, 1972.

\bibitem[She83a]{Sh87a} S.~Shelah.
\newblock Classification theory for nonelementary classes. {I}. the number of
  uncountable models of $\psi \in {L}_{\omega _{1}\omega }$ part {A}.
\newblock {\em Israel Journal of Mathematics}, 46:3:212--240, 1983.
\newblock Sh index 87a.

\bibitem[She83b]{Sh87b} S.~Shelah.
\newblock Classification theory for nonelementary classes. {II}. the number of
  uncountable models of $\psi \in {L}_{\omega _{1}\omega }$ part {B}.
\newblock {\em Israel Journal of Mathematics}, 46;3:241--271, 1983.
\newblock Sh index 87b.

\bibitem[She90]{Shelahbook2nd} S.~Shelah.
\newblock {\em Classification {T}heory and the {N}umber of {N}onisomorphic
  {M}odels}.
\newblock North-Holland, 1990.
\newblock second edition.

\bibitem[Shi71]{Shimura} G.~Shimura.
\newblock {\em Introduction to the Arithmetic Theory of Automorphic Functions}.
\newblock Iwanami Shoten and Princeton University Press, 1971.

\bibitem[TZ12]{TentZiegler} Katrin Tent and Martin Ziegler.
\newblock {\em A course in Model Theory}.
\newblock Lecture Notes in Logic. Cambridge University Press, 2012.

\bibitem[Vas18]{Vaseyqmaec} Sebastien Vasey.
\newblock Quasiminimal abstract elementary classes.
\newblock {\em Archive for Mathematical Logic}, 57:299--315, 2018.

\bibitem[Zil84]{Zilbericm} B.I. Zilber.
\newblock The structure of models of uncountably {c}ategorical {t}heories.
\newblock In {\em Proceedings of the International Congress of Mathematicians
  August 16-23, 1983, Warszawa}, pages 359--68. Polish Scientific Publishers,
  Warszawa, 1984.

\bibitem[Zil04]{Zilberpseudoexp} B.I. Zilber.
\newblock Pseudo-exponentiation on algebraically closed fields of
  characteristic 0.
\newblock {\em Annals of Pure and Applied Logic}, 132:67--95, 2004.

\bibitem[Zil05a]{Zilberparis} B.I. Zilber.
\newblock Analytic and pseudo-analytic structures.
\newblock In Rene Cori, Alexander Razborov, Stevo Todorcevic, and Carol Wood,
  editors, {\em Logic Colloquium 2000; Paris, France July 23-31, 2000},
  number~19 in Lecture Notes in Logic. Association of Symbolic Logic, 2005.

\bibitem[Zil05b]{Zilbercatex} B.I. Zilber.
\newblock A categoricity theorem for quasiminimal excellent classes.
\newblock In A.~Blass and Y.~Zhang, editors, {\em Logic and its
  {A}pplications}, volume 380 of {\em Contemporary Mathematics}, pages
  297--306. American Mathematical Society, Providence, RI, 2005.

\bibitem[Zil06]{Zilbercovers} B.I. Zilber.
\newblock Covers of the multiplicative group of an algebraically closed field
  of characteristic 0.
\newblock {\em Journal of the London Mathematical Society}, pages 41--58, 2006.

\bibitem[Zil10]{ZilberZariski} B.I. Zilber.
\newblock {\em Zariski Geometries: Geometry from the Logicians Point of View}.
\newblock Number 360 in London Math. Soc. Lecture Notes. London Mathematical
  Society, Cambridge University Press, 2010.

\bibitem[Zil22]{Zilbomin} B.I. Zilber.
\newblock Non-elementary categoricity and projective locally o-minimal classes.
\newblock on Zilber's webpage, 2022.

\end{thebibliography}
%

\newcommand{\etalchar}[1]{$^{#1}$}

\end{document}